\documentclass[11pt]{amsart}
\usepackage[T1]{fontenc}
\usepackage[utf8]{inputenc} % ok for pdflatex; remove if compiling with lualatex/xelatex
\usepackage[
  backend=biber,
  style=numeric,
  sorting=none,
  giveninits=false,
  doi=true,
  isbn=true,
  url=true,
  eprint=true
]{biblatex}

\usepackage{lmodern}
\usepackage{microtype}

\usepackage{amsmath, amssymb, amsfonts, amsthm}
\usepackage{mathtools}
\usepackage{enumitem}
\usepackage{hyperref}
\usepackage{cleveref}
\usepackage{tikz-cd}
\usepackage{geometry}
\usepackage{thmtools}
\usepackage{array}
\usepackage{caption}
\newcolumntype{C}[1]{>{\centering\arraybackslash}m{#1}}
\usepackage{float}
\usepackage{caption}

\hypersetup{
  colorlinks=true,
  linkcolor=blue,
  citecolor=blue,
  urlcolor=blue,
  pdftitle={Perverse Schober Models for Conifold Degenerations and Duality-Preserving Invariants},
  pdfauthor={Abdul Rahman}
}
\numberwithin{equation}{section}
\theoremstyle{plain}
\newtheorem{theorem}{Theorem}[section]
\newtheorem{proposition}[theorem]{Proposition}

\newtheorem{corollary}[theorem]{Corollary}
\newtheorem{conjecture}[theorem]{Conjecture}
\theoremstyle{definition}

\theoremstyle{remark}
\newtheorem{remark}{Remark}[section]

\newcommand{\Q}{\mathbb{Q}}
\newcommand{\Z}{\mathbb{Z}}
\newcommand{\C}{\mathbb{C}}

\newcommand{\Perv}{\mathrm{Perv}}

\newcommand{\Ext}{\mathrm{Ext}}

\newcommand{\var}{\mathrm{var}}
\newcommand{\rat}{\mathrm{rat}}

\newcommand{\HH}{\mathrm{HH}}
\newcommand{\Gr}{\mathrm{Gr}}
\DeclareMathOperator{\Cone}{Cone}

\DeclareMathOperator{\id}{id}

\title[Finite-Node Perverse Schobers and Corrected Extensions]{Finite-Node Perverse Schobers and Corrected Extensions for Conifold Degenerations}

\author{Abdul Rahman}
\thanks{Email: arahman@alum.howard.edu}
 % edit
 % edit

\begin{document}

\begin{abstract}
We study one-parameter conifold degenerations whose central fiber has finitely many ordinary double points. Working within a deliberately minimal finite-node bulk/localized-sector formalism, we identify the first categorical layer suggested by the corrected finite-node perverse extension and its mixed-Hodge-module package. Assuming that the local ordinary-double-point coupling pattern admits categorical realization in this finite-node setting, we formalize the corresponding local and finite-node data over a chosen bulk category, prove compatibility of their specified shadows with the corrected finite-node perverse extension established in earlier work, and isolate one localized categorical sector per node. We also extract a first finite combinatorial skeleton encoding the nodewise coupling pattern. The paper does not claim a universal perverse-schober theory for arbitrary singular Calabi--Yau degenerations, nor a categorical wall-crossing theory. Rather, it provides the foundational finite-node categorical formalization layer above the corrected perverse and mixed-Hodge-module packages in the conifold degeneration.
\end{abstract}

\maketitle

\tableofcontents
%=====================================================
\section{Introduction}

The purpose of the present paper is not to construct a universal external theory of perverse schobers for arbitrary conifold degenerations. Rather, we isolate the finite-node categorical data naturally suggested by the corrected perverse extension and its mixed-Hodge-module refinement, package that data into a deliberately minimal finite-node bulk/localized-sector formalism, and prove local-to-global compatibility results within that formalism. In this sense, the paper supplies the first categorical bulk/localized-sector layer attached to the finite-node conifold degeneration, while deliberately restricting to the weakest formalism needed for that purpose.

More precisely, the present paper is a conditional formalization-and-assembly paper. Assuming that the local ordinary-double-point coupling pattern admits categorical realization in the finite-node formalism, we formalize the corresponding local and finite-node data over a chosen bulk category, prove compatibility of their specified shadows with the corrected finite-node perverse extension, isolate one localized categorical sector per node, and extract a first finite combinatorial skeleton recording the nodewise coupling pattern. The paper does not claim a full Kapranov--Schechtman schober on a stratified base, nor a categorical wall-crossing theory; those remain downstream of the finite-node architecture organized here.

\subsection{Relation to earlier work and Hodge-theoretic framework}

The use of perverse sheaves in the study of singular Calabi--Yau spaces and conifold transitions goes back to earlier attempts to construct cohomological models that retain duality-theoretic features in the presence of singularities. Foundational work of Beilinson--Bernstein--Deligne, together with the linear-algebraic descriptions of perverse sheaves developed by MacPherson--Vilonen and Gelfand--MacPherson--Vilonen, provides the basic formalism for treating a space with a single singular stratum in terms of gluing and extension data \cite{BBD,MacPhersonVilonen1986,GMV1996}. In the conifold setting, such ideas have already appeared in the construction of perverse-sheaf models for corrected cohomology theories and in the study of singular Calabi--Yau threefolds arising in string theory \cite{RahmanATMP,BanaglIS,BanaglBudurMaxim}.

In \cite{RahmanSchoberPaper}, we approached a one-parameter conifold degeneration from the nearby/vanishing-cycle side and isolated the canonical perverse sheaf
\[
\mathcal P := \Cone(\var_F)[-1]
\]
as a natural perverse object on the singular fiber in the ordinary double point case. In that setting, the main point was the existence of a canonical perverse extension determined by the variation morphism between vanishing and nearby cycles, together with its relation to the rank-one local contribution coming from Picard--Lefschetz monodromy. The emphasis in \cite{RahmanSchoberPaper} was primarily sheaf-theoretic: the object $\mathcal P$ was constructed and analyzed inside the category of rational perverse sheaves, and the mixed-Hodge-theoretic refinement was left as a further direction.

The present paper takes up that Hodge-theoretic direction using Saito's theory of mixed Hodge modules \cite{SaitoMHM,SaitoDuality}. For a complex algebraic variety $X$, Saito constructs an abelian category $MHM(X)$ together with an exact and faithful realization functor
\[
\rat \colon MHM(X)\to \Perv(X),
\]
to the category of rational perverse sheaves \cite{SaitoMHM}. In particular, the nearby-cycle and vanishing-cycle functors admit lifts to the mixed-Hodge-module setting and are compatible, under $\rat$, with the corresponding functors on perverse sheaves \cite{SaitoMHM}. This is the basic formal mechanism that allows one to compare the Hodge-theoretic degeneration data of a family with the canonical perverse sheaf attached to its singular fiber.

A second ingredient from Saito, especially relevant here, is the divisor-case gluing formalism. If $Y:=g^{-1}(0)$ is a principal divisor in a variety $X$ and $U:=X\setminus Y$, then mixed Hodge modules along $Y$ may be described in terms of data on $U$, data on $Y$, and morphisms relating nearby-cycle information, with compatibility governed by the nilpotent monodromy operator \cite{SaitoMHM}. Since the central fiber
\[
X_0:=\pi^{-1}(0)
\]
of a one-parameter degeneration is a principal divisor, this formalism provides the natural structural setting in which to reinterpret the canonical perverse extension arising from nearby cycles. For the purposes of the present paper, we use only the general consequences of Saito's theory needed for this reinterpretation: the existence of the categories $MHM(X)$, the exact faithful functor $\rat$, the mixed-Hodge-module nearby- and vanishing-cycle functors, and the divisor-case gluing formalism \cite{SaitoMHM}.

An important precedent for this point of view is the work of Banagl--Budur--Maxim \cite{BanaglBudurMaxim}. In their setting, for a projective hypersurface with an isolated singularity, they construct a perverse sheaf whose hypercohomology computes the intersection-space cohomology and show that this perverse sheaf underlies a mixed Hodge module, so that its hypercohomology inherits canonical mixed Hodge structures \cite{BanaglBudurMaxim}. Under an additional semisimplicity hypothesis on the local monodromy for the eigenvalue $1$, they further obtain a splitting of the nearby-cycle perverse sheaf in which their intersection-space complex appears as a direct summand \cite{BanaglBudurMaxim}. Although the object studied in \cite{BanaglBudurMaxim} is different from the canonical perverse extension considered here, that paper provides a useful model for how a perverse sheaf built from degeneration data can be placed in a mixed-Hodge-module framework without identifying the two categories.

The relation between \cite{BanaglBudurMaxim} and the present paper is therefore one of method rather than of direct equivalence of objects. Their intersection-space complex is designed to recover intersection-space cohomology for isolated hypersurface singularities, whereas our object of study is the canonical perverse extension
\[
\mathcal P := \Cone(\var_F)[-1]
\]
attached to a conifold degeneration with a single ordinary double point. What the two settings share is the central role of nearby and vanishing cycles, the perverse-sheaf description of the local singular contribution, and the possibility of passing to a mixed-Hodge-module refinement. This makes \cite{BanaglBudurMaxim} a natural reference point for the Hodge-theoretic direction pursued here.

The present paper extends \cite{RahmanSchoberPaper} in a more limited but structurally important direction. It does not re-establish the existence and basic properties of the perverse sheaf $\mathcal P$, nor does it construct a universal categorical theory of conifold degenerations. Rather, it places the finite-node corrected package into a minimal categorical organizational framework above the previously established perverse and mixed-Hodge-module layers. In this hierarchy, the earlier papers supplied the corrected perverse object and its mixed-Hodge-module refinement target, while the present paper supplies the first finite-node categorical bulk/localized-sector formalization compatible with those earlier layers.

\subsection{Physical and categorical motivation}

Conifold degenerations occupy a distinguished position in the geometry of Calabi--Yau threefolds and in string theory. In the ordinary double point case, the degeneration is governed by a vanishing three-sphere in the Milnor fiber, and the associated local monodromy on middle homology is given by the Picard--Lefschetz formula
\[
T(\alpha)=\alpha+(\alpha\cdot\delta)\delta,
\]
where \(\delta\) denotes the vanishing cycle \cite{MilnorSingularPoints,DimcaSheaves}.
Thus the singular fiber carries a rank-one local correction term controlled by the vanishing sphere and its monodromy.

In Strominger's physical interpretation of the conifold transition, the collapse of this three-cycle gives rise to an additional light BPS state, and the singular behavior of the effective moduli space is resolved only after this extra degree of freedom is taken into account \cite{StromingerConifold}. From this perspective, the conifold point is not merely a singular limit of the family, but a place where a new rank-one sector becomes visible. The ordinary double point case is therefore an especially useful model for comparing geometric, sheaf-theoretic, and Hodge-theoretic manifestations of the same local phenomenon.

The categorical counterpart is connected to homological mirror symmetry, where the Picard--Lefschetz transformation is mirrored by a spherical object whose associated spherical twist induces a rank-one reflection on additive invariants such as the Grothendieck group \cite{SeidelThomas}. More generally, Kapranov and Schechtman introduced perverse schobers as categorical analogues of perverse sheaves, with local monodromy governed by spherical functors and their twists \cite{KapranovSchechtman}. In this sense, the ordinary double point provides a setting in which topological monodromy, limiting Hodge theory, and categorical monodromy all exhibit the same rank-one correction mechanism.

The purpose of the present paper is to isolate the finite-node categorical content suggested by that picture while remaining strictly within a minimal formalism. We do not construct a full spherical-functor model, nor a genuine perverse schober in the full Kapranov--Schechtman sense. Rather, we use these categorical precedents as motivation for organizing one bulk sector together with one localized sector per node in a finite-node conifold degeneration. On the Hodge-theoretic side, the key structural input remains Saito's divisor-case gluing formalism for mixed Hodge modules, which identifies the natural framework in which a fuller mixed-Hodge-module refinement should eventually be constructed \cite[Prop.~0.3]{SaitoMHM}.

\subsection{Main results}

We now summarize the main theorem package of the paper. All categorical existence and uniqueness statements below are proved within the finite-node formalism introduced in Section~3. In particular, the present paper works relative to three standing inputs: a chosen bulk category, assumed categorical realizability of the local ordinary-double-point coupling pattern, and a specified notion of perverse shadow inside the formalism. No explicit dg- or triangulated-category model for the bulk or localized sectors is constructed in the present paper. 

A basic motivating model for this realizability hypothesis is the local Kapranov--Schechtman picture associated with a spherical functor on a disk with one marked point \cite{KapranovSchechtman}, where the fiber category is triangulated and the local monodromy is given by the corresponding spherical twist. The present finite-node formalism is designed to be compatible with that local model, although no full identification theorem with that framework is proved here.

\begin{theorem}[Conditional local ODP formalization] \label{thm:cond-local-odp-formalization}
For a one-parameter degeneration with a single ordinary double point \(p\), assume that the local ordinary-double-point coupling pattern admits categorical realization in the finite-node formalism. Then there exists a local ODP datum refining that local coupling pattern and equipped with specified corrected local perverse shadow.
\end{theorem}

\begin{theorem}[Conditional finite-node assembly] \label{thm:cond-finite-node-assembly}
Let
\[
\Sigma=\{p_1,\dots,p_r\}\subset X_0
\]
be the finite node set of the central fiber. Assume that the local ordinary-double-point coupling patterns admit categorical realizations in the finite-node setting, and that the global smooth sector is equipped with a chosen bulk category. Then the resulting local data assemble into a finite-node datum with one localized categorical sector per node.
\end{theorem}

\begin{theorem}[Global shadow compatibility] \label{thm:global-shadow-compatibility}
Let \(S_\Sigma\) be the finite-node datum produced by the assembly theorem. Then its specified global shadow agrees, within the finite-node formalism, with the corrected finite-node perverse extension
\[
0 \to IC_{X_0} \to P \to \bigoplus_{k=1}^r i_{k*}\mathbf Q_{\{p_k\}} \to 0
\]
constructed in the earlier perverse and mixed-Hodge-module papers.
\end{theorem}

\begin{remark}
The preceding theorem is a compatibility theorem rather than a universal decategorification theorem. In the present formalism, the shadow is specified as part of the datum; the content is that the local-to-global assembly recovers, at the level of specified shadow, the previously established corrected finite-node perverse extension.
\end{remark}

\begin{theorem}[Localized sectors and first combinatorial skeleton] \label{thm:localized-sectors-first-quiver-shadow}
The finite-node datum contains one distinguished localized categorical sector for each node of the degeneration. These sectors determine a finite combinatorial skeleton compatible with the nodewise organization of the corrected finite-node perverse extension and providing a first precursor of later quiver and wall-crossing constructions.
\end{theorem}

Taken together, these results show that the corrected finite-node extension admits a finite-node categorical bulk/localized-sector packaging whose specified shadow is compatible with the corrected perverse layer and with the previously established mixed-Hodge-module refinement target.

\subsection{Proof strategy}

The argument proceeds in four steps. First, we formalize the local ordinary double point model inside the finite-node framework. This local step fixes the bulk category, the localized node category, the attachment functors, and the corrected local perverse shadow, always under the stated realizability hypothesis.

Second, we pass from local to global by organizing the finite family of local ODP data over the finite node set and assembling them relative to a chosen bulk category on the smooth sector. Third, we verify that the specified global shadow of the assembled finite-node datum agrees with the previously established corrected finite-node perverse extension. Fourth, we isolate the localized categorical sectors, their coupling pattern, and the resulting finite combinatorial skeleton, and we record the corresponding uniqueness statements within the formalism.

In the terminology of the present paper, the \emph{bulk sector} refers to the categorical datum associated with the smooth geometric part of the degeneration, while a \emph{localized sector} refers to the categorical datum attached to an individual ordinary double point. The \emph{shadow} of a finite-node datum means its specified perverse-sheaf-theoretic image, which in the present setting is required to agree with the corrected finite-node perverse extension already established in earlier work.

Thus the contribution of the present paper is not a full categorical wall-crossing theory, but the conditional formalization and organization of the finite-node categorical datum on which such later theories may act.

\subsection{Scope and organization}

The paper is confined to the case of finitely many ordinary double points. This is the natural next setting after the single-node corrected extension of~\cite{RahmanSchoberPaper}, the nearby-cycle and limiting mixed-Hodge-structure comparison developed in the companion work, and the finite-node mixed-Hodge-module lift established in the earlier finite-node paper. In particular, we do not attempt here a full treatment of arbitrary higher-dimensional singular strata, nor a construction in a universal external theory of perverse schobers. The present paper should instead be read as a finite-node categorical formalization layer above the corrected perverse and mixed-Hodge-module packages.

The finite-node datum introduced below is not claimed to be a perverse schober in the full Kapranov--Schechtman sense. It is a deliberately minimal tuple-level bulk/localized-sector formalism designed to organize the finite-node conifold package at its first categorical layer. The term ``schober datum'' is therefore used only in this restricted formal sense whenever convenient; no claim is made here of a full schober-theoretic construction on a stratified base.

Section~2 recalls the geometric and Hodge-theoretic background. Section~3 introduces the finite-node formalism and defines the local ODP datum. Section~4 isolates the finite-node corrected package and its bulk/localized-sector organization. Section~5 records the finite-node corrected extension as the fixed source object for the later categorical layer. Section~6 proves the corresponding global mixed-Hodge-module gluing and its compatibility under realization. Section~7 records the single-node uniqueness and Verdier self-duality properties of the corrected perverse object. Section~8 assembles the main theorem package. Section~9 discusses the deferred K\"ahler-package question, and Section~10 records future directions.

%==================================================================
\section{Limiting Mixed Hodge Structures}
\label{sec:LMHS}

Let
\[
\pi:\mathcal X \to \Delta
\]
be a projective morphism of complex algebraic varieties, smooth over
\(\Delta^*=\Delta\setminus\{0\}\), with smooth fiber \(X_t=\pi^{-1}(t)\) for
\(t\neq 0\) and singular central fiber \(X_0=\pi^{-1}(0)\).
For each \(k\), the cohomology groups \(H^k(X_t,\Q)\) form a polarized variation
of Hodge structure over \(\Delta^*\) \cite{Schmid}. After passing to the universal
cover of \(\Delta^*\), or equivalently after choosing a reference fiber and a branch
of the logarithm, one obtains the corresponding limiting mixed Hodge structure in the
sense of Schmid and Steenbrink \cite{Schmid,SteenbrinkLimits}.

\subsection{Monodromy and the limiting mixed Hodge structure}

Fix \(k\ge 0\). Parallel transport around a positively oriented simple loop about
\(t=0\) defines the monodromy operator
\[
T:H^k(X_t,\Q)\to H^k(X_t,\Q).
\]
By the monodromy theorem, \(T\) is quasi-unipotent \cite[Thm.~6.16]{Schmid}.
After a finite base change \(t\mapsto t^m\), one may assume that \(T\) is unipotent.
In that case one defines
\[
N:=\log T
   = (T-\id)-\frac{1}{2}(T-\id)^2+\frac{1}{3}(T-\id)^3-\cdots,
\]
which is nilpotent.

Associated with \(N\) is the monodromy weight filtration
\[
W(N)_\bullet
\]
on the limiting cohomology group \(H^k_{\lim}:=H^k(X_\infty,\Q)\), characterized by
the usual conditions
\[
N\,W(N)_\ell \subseteq W(N)_{\ell-2},
\qquad
N^j:\Gr^{W(N)}_{k+j}H^k_{\lim}\xrightarrow{\sim}\Gr^{W(N)}_{k-j}H^k_{\lim}
\quad (j\ge 0),
\]
where the filtration is centered at degree \(k\) \cite{Schmid,SteenbrinkLimits}.
Together with the limiting Hodge filtration \(F^\bullet_\infty\) obtained from the
nilpotent orbit theorem, this yields the limiting mixed Hodge structure
\[
\bigl(H^k_{\lim},W(N)_\bullet,F^\bullet_\infty\bigr)
\]
on \(H^k_{\lim}\) \cite{Schmid,SteenbrinkLimits}.

\subsection{Nearby cycles and vanishing cycles}

Let \(i:X_0\hookrightarrow \mathcal X\) denote the inclusion of the central fiber.
For \(K\in D^b_c(\mathcal X,\Q)\), the nearby-cycle and vanishing-cycle functors
\[
\psi_\pi K,\qquad \phi_\pi K
\]
fit into standard distinguished triangles in \(D^b_c(X_0,\Q)\); see, for example,
\cite[§4.2]{DimcaSheaves}. In particular, there is a functorial distinguished triangle
\begin{equation}
i^*K \longrightarrow \psi_\pi K \longrightarrow \phi_\pi K \overset{+1}{\longrightarrow}.
\label{eq:nearby-vanishing-triangle}
\end{equation}
Applying hypercohomology to \eqref{eq:nearby-vanishing-triangle} gives a long exact sequence
\begin{equation}
\cdots \to H^m(X_0,i^*K)\to \HH^m(X_0,\psi_\pi K)\to
\HH^m(X_0,\phi_\pi K)\to H^{m+1}(X_0,i^*K)\to\cdots .
\label{eq:nearby-vanishing-les}
\end{equation}

When \(K=\Q_{\mathcal X}\), the hypercohomology of the nearby-cycle complex identifies
with the cohomology of the canonical nearby fiber:
\[
\HH^m(X_0,\psi_\pi \Q_{\mathcal X}) \cong H^m(X_\infty,\Q),
\]
compatibly with monodromy and with the limiting mixed Hodge structure
\cite{SteenbrinkLimits,DimcaSheaves}. Thus the long exact sequence
\eqref{eq:nearby-vanishing-les} relates the cohomology of the central fiber, the limiting
cohomology, and the vanishing cohomology.

\subsection{The ordinary double point case}

We now specialize to a one-parameter degeneration of complex threefolds whose central fiber
\(X_0\) has a single ordinary double point \(p\). The Milnor fiber of an ordinary double
point in complex dimension \(3\) has the homotopy type of \(S^3\), hence its reduced
cohomology is one-dimensional in degree \(3\) and vanishes in all other degrees
\cite{MilnorSingularPoints,DimcaSheaves}. Equivalently, the local vanishing cohomology is
concentrated in the middle degree and has rank one.

Accordingly, the vanishing-cycle complex for the degeneration is supported at \(p\), and its
only nontrivial local contribution occurs in the middle degree. It follows from
\eqref{eq:nearby-vanishing-les} that the difference between the limiting cohomology and the
cohomology of the central fiber is governed by a single rank-one vanishing contribution.
More precisely, after the conventional shift placing nearby and vanishing cycles in the
perverse heart on \(X_0\), the corresponding vanishing-cycle perverse sheaf is a skyscraper
object of rank one at \(p\). This is the sheaf-theoretic manifestation of the Picard--Lefschetz
correction term.

For the purposes of the present paper, we will only use the following consequence: in the
ordinary double point case, the nearby/vanishing-cycle triangle carries exactly one local
vanishing degree of freedom, and this local rank-one contribution is the source of the
canonical perverse extension studied below.

\subsection{Mixed Hodge modules and realization}

The Hodge-theoretic refinement of nearby and vanishing cycles is provided by Saito's theory of
mixed Hodge modules \cite{SaitoMHM,SaitoDuality}. For every complex algebraic variety \(Y\),
Saito constructs an abelian category \(MHM(Y)\) together with an exact and faithful realization
functor
\[
\rat:MHM(Y)\to \Perv(Y;\Q)
\]
to rational perverse sheaves \cite{SaitoMHM}. Moreover, the nearby-cycle and vanishing-cycle
functors admit lifts to the mixed-Hodge-module setting and are compatible with the underlying
rational perverse sheaves under \(\rat\) \cite{SaitoMHM}.

In particular, if \(\pi:\mathcal X\to\Delta\) is as above, then the nearby-cycle mixed Hodge
module \(\psi_\pi^H(\Q_{\mathcal X})\) carries the limiting mixed Hodge structure on cohomology,
while its underlying rational perverse sheaf is the usual nearby-cycle perverse sheaf.
Likewise, the vanishing-cycle mixed Hodge module \(\phi_\pi^H(\Q_{\mathcal X})\) refines the
usual vanishing-cycle object. This formalism is the basic reason that one may compare the
canonical perverse extension on \(X_0\) with Hodge-theoretic degeneration data.

The following proposition records the only general fact from this formalism that we will use
in the next section.

\begin{proposition}
\label{prop:mhm-nearby-cycles}
Let \(\pi:\mathcal X\to\Delta\) be a one-parameter degeneration, and let
\(K\in D^bMHM(\mathcal X)\). Then the mixed-Hodge-module nearby-cycle and vanishing-cycle
functors fit into the corresponding functorial triangles in \(D^bMHM(X_0)\), and after applying
\(\rat\) one recovers the standard nearby/vanishing-cycle triangles in
\(D^b_c(X_0,\Q)\). In particular, applying hypercohomology to the nearby-cycle mixed Hodge
module recovers the limiting mixed Hodge structure on cohomology.
\end{proposition}

\begin{proof}
The existence of nearby and vanishing cycle functors in the category of mixed Hodge modules,
together with their compatibility with the underlying rational perverse sheaves, is part of
Saito's formalism; see \cite{SaitoMHM,SaitoDuality}. The identification of the hypercohomology
of nearby cycles with limiting cohomology, endowed with its limiting mixed Hodge structure, is
standard in the work of Steenbrink and Saito; see \cite{SteenbrinkLimits,SaitoMHM}.
\end{proof}

\medskip

Proposition~\ref{prop:mhm-nearby-cycles} does \emph{not} by itself identify a specific extension
class in the category of mixed Hodge structures with a specific extension class in the category
of perverse sheaves. Rather, it shows that both the perverse-sheaf-theoretic and the
Hodge-theoretic constructions are functorially derived from the same nearby-cycle formalism.
The comparison with the canonical perverse extension attached to the ordinary double point
degeneration will be carried out in the next section.
%======================================================================
%======================================================================
\section{The single-node case}
\label{sec:ODP}

We now specialize to the basic local model for the later finite-node theory. Let
\[
\pi:\mathcal X\to\Delta
\]
be a one-parameter degeneration whose general fiber \(X_t\) is a smooth complex threefold and whose
central fiber \(X_0\) has a single ordinary double point
\[
p\in X_0.
\]
Write
\[
U:=X_0\setminus\{p\},
\qquad
j:U\hookrightarrow X_0,
\qquad
i:\{p\}\hookrightarrow X_0
\]
for the smooth locus and the natural inclusions.

The purpose of this section is threefold. First, we recall the local Picard--Lefschetz and
vanishing-cycle structure of an ordinary double point. Second, we state the canonical corrected
perverse extension in this setting. Third, we place that extension into Saito's mixed-Hodge-module
framework and isolate, as sharply as possible, the remaining gluing problem needed for a fully
internal Hodge-theoretic refinement. The point is not to claim that this refinement has already
been carried out in complete generality in the present paper, but to show precisely how far the
nearby-cycle formalism already determines it.

\subsection{Vanishing cycles and Picard--Lefschetz}
\label{subsec:ODP-PL}

For an ordinary double point in complex dimension \(3\), the Milnor fiber has the homotopy type of
\(S^3\). Equivalently, its reduced cohomology is one-dimensional in degree \(3\) and vanishes in all
other degrees \cite{MilnorSingularPoints,DimcaSheaves}. Thus the local vanishing cohomology is
concentrated in the middle degree and has rank one.

Let
\[
\delta\in H_3(X_t,\Z)
\]
be a vanishing cycle. The local monodromy transformation about \(t=0\) acts on middle homology by
the Picard--Lefschetz formula. In the present situation this action is rank one and is completely
controlled by the vanishing sphere \(\delta\); see \cite[Ch.~11]{MilnorSingularPoints} and
\cite[\S4.1]{DimcaSheaves}. For the arguments below, the essential consequence is that the
vanishing-cycle contribution is both one-dimensional and supported at the singular point.

This rank-one local structure is the decisive simplification in the ordinary double point case. It is
what makes possible a canonical corrected perverse extension with point-supported singular quotient,
and it is also what later allows one to formulate a concrete mixed-Hodge-module gluing problem in
the divisor formalism.

\subsection{The perverse extension}
\label{subsec:ODP-pervext}

Let
\[
F:=\Q_{\mathcal X}[3].
\]
Since \(\dim_\C X_t=3\), this is the natural shifted object for nearby and vanishing cycles in the
perverse normalization. Consider the variation morphism
\[
\var_F:\phi_\pi(F)\to \psi_\pi(F),
\]
and define
\[
\mathcal P:=\Cone\!\bigl(\var_F:\phi_\pi(F)\to\psi_\pi(F)\bigr)[-1].
\]

In the ordinary double point case, the vanishing-cycle perverse sheaf is supported at \(p\) and has
one-dimensional stalk there. Consequently, the cone construction produces an extension of the
intersection complex by a point-supported rank-one contribution. More precisely, from
\cite{RahmanSchoberPaper} one has a short exact sequence in \(\Perv(X_0;\Q)\)
\begin{equation}
0\longrightarrow IC_{X_0}
\longrightarrow \mathcal P
\longrightarrow i_*\Q_{\{p\}}
\longrightarrow 0,
\label{eq:perverse-extension-ODP}
\end{equation}
where
\[
IC_{X_0}:=j_{!*}\Q_U[3].
\]

This exact sequence is the first appearance of the corrected extension package in the single-node
setting. It exhibits \(\mathcal P\) as the canonical perverse object obtained by adjoining the
rank-one singular contribution to the middle extension of the constant sheaf on the smooth locus.
The relevant sheaf-theoretic background comes from the standard theory of nearby and vanishing
cycles, together with the description of perverse sheaves with isolated singularities via recollement
and zig-zags \cite{BBD,MacPhersonVilonen1986,GMV1996,DimcaSheaves}.

\subsection{Mixed Hodge modules and nearby cycles}
\label{subsec:ODP-MHM}

We now place the preceding construction into Saito's framework of mixed Hodge modules. For a
complex algebraic variety \(Y\), Saito constructs an abelian category \(MHM(Y)\) together with an
exact and faithful realization functor
\[
\rat:MHM(Y)\to \Perv(Y;\Q)
\]
to rational perverse sheaves \cite{SaitoMHM}. Nearby-cycle and vanishing-cycle functors lift to the
mixed-Hodge-module setting and are compatible, under \(\rat\), with the corresponding functors on
perverse sheaves \cite{SaitoMHM,SaitoDuality}.

Applied to the degeneration \(\pi:\mathcal X\to\Delta\), this yields mixed Hodge modules
\[
\psi_\pi^H(\Q_{\mathcal X}[3]),
\qquad
\phi_\pi^H(\Q_{\mathcal X}[3])
\]
on \(X_0\), together with the corresponding canonical morphisms in \(MHM(X_0)\). Their images
under \(\rat\) are the usual nearby- and vanishing-cycle perverse sheaves associated with
\[
F=\Q_{\mathcal X}[3].
\]
Thus the same formalism that produces the corrected perverse object also carries a mixed-Hodge
theoretic refinement at the level of nearby and vanishing cycles.

The next proposition records the exact comparison needed here. Its content is deliberately modest.
It does not identify a specific extension class in mixed Hodge structures with a specific extension
class in perverse sheaves. Rather, it shows that the local rank-one correction on the perverse side
and the local rank-one vanishing contribution on the Hodge-theoretic side arise from the same
nearby-cycle/vanishing-cycle mechanism.

\begin{proposition}
\label{prop:ODP-compatible}
Let \(\pi:\mathcal X\to\Delta\) be a one-parameter degeneration whose central fiber \(X_0\) has a
single ordinary double point. Then the canonical perverse sheaf \(\mathcal P\) defined in
\eqref{eq:perverse-extension-ODP} is functorially related to the nearby- and vanishing-cycle
formalism in Saito's category of mixed Hodge modules through the realization functor
\[
\rat:MHM(X_0)\to\Perv(X_0;\Q).
\]
In particular, the point-supported rank-one contribution in \eqref{eq:perverse-extension-ODP} and
the rank-one vanishing contribution in the limiting mixed Hodge structure both arise from the same
nearby-cycle/vanishing-cycle construction.
\end{proposition}

\begin{proof}
By Saito's theory, nearby-cycle and vanishing-cycle functors are defined for mixed Hodge modules and
are compatible, after applying \(\rat\), with the corresponding functors on rational perverse sheaves
\cite{SaitoMHM,SaitoDuality}. Therefore the mixed-Hodge-module nearby and vanishing cycle objects
attached to \(\Q_{\mathcal X}[3]\) determine, under realization, the usual nearby- and vanishing-cycle
perverse sheaves on \(X_0\).

In the ordinary double point case, the Milnor fiber has reduced cohomology of rank one in degree
\(3\) and trivial reduced cohomology in all other degrees
\cite{MilnorSingularPoints,DimcaSheaves}. Hence the local vanishing-cycle contribution is rank one.
On the perverse-sheaf side this yields the point-supported quotient
\[
i_*\Q_{\{p\}}
\]
in \eqref{eq:perverse-extension-ODP}, while on the Hodge-theoretic side the same local
vanishing-cycle data contributes the rank-one vanishing part of the limiting mixed Hodge structure.
Thus both constructions are functorially derived from the same nearby-cycle formalism.
\end{proof}

\subsection{Remarks on the Hodge-theoretic comparison}
\label{subsec:ODP-remarks}

Proposition~\ref{prop:ODP-compatible} should be read as a common-origin statement, not as a full
comparison theorem for extension classes. In particular, it does \emph{not} assert a canonical
identification between an extension class in
\[
\Ext^1_{\mathrm{MHS}}
\]
and the extension class of \eqref{eq:perverse-extension-ODP} in
\[
\Ext^1_{\Perv(X_0;\Q)}.
\]
A theorem of that strength would require a more detailed comparison between the extension data
visible under realization, the corresponding hypercohomology packages, and the relevant filtered
structures.

What Proposition~\ref{prop:ODP-compatible} does establish is the precise point needed for the
present paper: the canonical corrected perverse extension and the Hodge-theoretic degeneration
data are produced by the same nearby- and vanishing-cycle formalism. In this sense the
mixed-Hodge-module picture should be viewed as a refinement target for the corrected perverse
object, not as an unrelated parallel construction.

\begin{remark}
A complete Hodge-theoretic refinement of \eqref{eq:perverse-extension-ODP} would consist of an
object
\[
\mathcal P^H\in MHM(X_0)
\]
fitting into an exact sequence
\[
0\to IC^H_{X_0}\to \mathcal P^H\to i_*\Q^H_{\{p\}}(-1)\to 0
\]
whose image under
\[
\rat:MHM(X_0)\to\Perv(X_0;\Q)
\]
is the perverse extension \eqref{eq:perverse-extension-ODP}. By Saito's divisor-case gluing
formalism for the principal divisor \(X_0=\pi^{-1}(0)\), the construction of such an object reduces
to the explicit identification of gluing data
\[
(\mathcal M',\mathcal M'',u,v)
\]
satisfying
\[
vu=N
\]
in the sense of \cite[Prop.~0.3]{SaitoMHM}. The present paper isolates this gluing problem sharply,
but does not yet solve it in full generality.
\end{remark}

\subsection{Mixed-Hodge-module refinement and the gluing problem}
\label{subsec:ODP-gluing}

The previous subsection isolates the exact point at which the nearby-cycle comparison stops. A
stronger statement would not merely say that the perverse and Hodge-theoretic constructions arise
from the same formalism, but would explicitly construct a mixed-Hodge-module object on \(X_0\)
whose underlying rational perverse sheaf is the corrected extension \(\mathcal P\).

Concretely, one would like to construct an object
\[
\mathcal P^H \in MHM(X_0)
\]
fitting into an exact sequence
\begin{equation}
0 \longrightarrow IC^H_{X_0}
\longrightarrow \mathcal P^H
\longrightarrow i_*\Q^H_{\{p\}}(-1)
\longrightarrow 0,
\label{eq:MHM-ODP-extension}
\end{equation}
such that
\[
\rat(\mathcal P^H)\cong \mathcal P
\]
and such that \eqref{eq:MHM-ODP-extension} refines the perverse extension
\eqref{eq:perverse-extension-ODP}. Here \(IC^H_{X_0}\) denotes the Hodge-module intersection
complex on \(X_0\), while \(i_*\Q^H_{\{p\}}(-1)\) is the point-supported mixed Hodge module
encoding the rank-one local vanishing contribution.

Saito's divisor-case gluing formalism provides the natural framework for this construction. If
\(Y=g^{-1}(0)\) is a principal divisor in a complex algebraic variety \(X\), then mixed Hodge modules
along \(Y\) may be described in terms of gluing data consisting of an object on \(X\setminus Y\), an
object on \(Y\), and morphisms
\[
u,\;v
\]
satisfying the relation
\[
vu=N,
\]
where \(N\) is the nilpotent logarithm of the unipotent part of monodromy
\cite[Prop.~0.3]{SaitoMHM}. In the present ordinary-double-point setting, the local vanishing
contribution is rank one, so the gluing problem is as simple as it can be while still remaining
genuinely nontrivial.

The mathematical challenge is therefore not to guess the existence of a Hodge-theoretic refinement,
but to verify explicitly that the corrected extension data can be organized into divisor-gluing data
of the required form. That is the real content of the single-node gluing problem. Put differently:
the common nearby-cycle origin established above is already enough to identify the correct
Hodge-theoretic target, but a full refinement theorem requires one additional step, namely the
construction of gluing data whose realization is exactly the corrected perverse extension.

This gluing problem is the point at which the present paper reaches the boundary of what can be
obtained from common-origin arguments alone. It is also the natural local model for the later
finite-node theory, where one must solve the same problem with a finite direct sum of point-supported
local contributions.
%=============================================================
%======================================================================
\section{The finite-node case}
\label{sec:finite-node}

We now pass from the single ordinary double point to the case of finitely many nodes. Let
\[
\pi:\mathcal X\to\Delta
\]
be a one-parameter degeneration whose central fiber \(X_0\) has finite singular set
\[
\Sigma=\{p_1,\dots,p_r\}\subset X_0,
\]
and assume that each \(p_k\) is an ordinary double point. Write
\[
U:=X_0\setminus \Sigma,
\qquad
j:U\hookrightarrow X_0,
\qquad
i_k:\{p_k\}\hookrightarrow X_0
\]
for the smooth locus and the natural inclusions.

The finite-node setting is the natural next step after the single-node theory. On the one hand, each
ordinary double point still contributes a rank-one local vanishing sector, so the local building
blocks remain formally simple. On the other hand, the global geometry is now assembled from a
finite family of singular points over one common bulk sector. This makes the finite-node problem a
genuine local-to-global problem rather than a trivial repetition of the single-node case.

The goal of the present section is to formulate the finite-node corrected package at the perverse and
mixed-Hodge-module levels, and to isolate the global gluing problem that later becomes decisive.
The emphasis is again on precision of structure rather than on overstatement: we identify the exact
finite-node package forced by nearby and vanishing cycles, and we explain how far this package is
already determined by local ODP data together with global assembly.

\subsection{Local ordinary-double-point blocks at each node}
\label{subsec:finite-local-blocks}

For each node \(p_k\in \Sigma\), choose a sufficiently small analytic neighborhood
\[
X_{0,\mathrm{loc},k}\subset X_0
\]
containing no other singular point. Since \(\Sigma\) is finite, these neighborhoods may be chosen
pairwise disjoint. Each \(X_{0,\mathrm{loc},k}\) is a local ordinary double point model of the type
treated in Section~\ref{sec:ODP}.

Accordingly, for each \(k\) one has a local corrected perverse extension
\[
0\longrightarrow IC_{X_{0,\mathrm{loc},k}}
\longrightarrow \mathcal P_{\mathrm{loc},k}
\longrightarrow i_{k*}\Q_{\{p_k\}}
\longrightarrow 0,
\]
and the corresponding mixed-Hodge-module gluing problem asks for a refinement
\[
0\longrightarrow IC^H_{X_{0,\mathrm{loc},k}}
\longrightarrow \mathcal P^H_{\mathrm{loc},k}
\longrightarrow i_{k*}\Q^H_{\{p_k\}}(-1)
\longrightarrow 0.
\]

Thus each node contributes the same formal local pattern as in the single-node case: a bulk
intersection-complex term together with a rank-one point-supported singular contribution. The new
issue is no longer the existence of the local ODP block itself, but the way in which finitely many
such blocks assemble over the common smooth sector \(U\).

\subsection{The corrected finite-node perverse extension}
\label{subsec:finite-perverse}

The perverse-theoretic finite-node package was identified in the earlier work as the corrected
finite-node extension
\begin{equation}
0\longrightarrow IC_{X_0}
\longrightarrow \mathcal P_\Sigma
\longrightarrow \bigoplus_{k=1}^r i_{k*}\Q_{\{p_k\}}
\longrightarrow 0.
\label{eq:finite-perv-extension}
\end{equation}
Here
\[
IC_{X_0}=j_{!*}\Q_U[3]
\]
is the middle extension of the shifted constant sheaf on the smooth locus, while the quotient
\[
Q_\Sigma:=\bigoplus_{k=1}^r i_{k*}\Q_{\{p_k\}}
\]
records the finite point-supported singular contribution.

The meaning of \eqref{eq:finite-perv-extension} is exactly parallel to the single-node case, but with
one important global change. The singular quotient is now a finite direct sum of local rank-one
sectors, one for each ordinary double point, while the bulk term remains a single global object.
Thus the finite-node corrected package already exhibits a bulk/localized-sector architecture:
\[
\text{one global bulk object}
\quad+\quad
\text{finitely many localized singular sectors}.
\]

This is the basic finite-node pattern that all later refinements must respect. In particular, any
mixed-Hodge-module lift, categorical enhancement, combinatorial skeleton, or later wall-crossing structure
must refine \eqref{eq:finite-perv-extension} rather than replace it.

\subsection{Nearby cycles and the finite singular quotient}
\label{subsec:finite-nearby}

The finite-node corrected extension is again governed by nearby and vanishing cycles. Since every
node is an ordinary double point, the local vanishing-cycle contribution at each \(p_k\) is rank one.
Therefore the finite singular contribution is point-supported and naturally indexed by the node set
\(\Sigma\).

At the level of nearby-cycle formalism, the key point is that the local rank-one ordinary-double-point
contributions coexist over one common degeneration. The finite singular quotient appearing in
\eqref{eq:finite-perv-extension} is therefore not an arbitrary sum of skyscraper terms, but the
precise singular package selected by the nearby- and vanishing-cycle formalism in the finite-node
setting.

This is the finite counterpart of the local statement in Section~\ref{sec:ODP}: the same formalism
that produces the corrected perverse extension also determines the singular quotient. The difference
is that in the finite-node case the quotient carries one summand for each node, so the geometry of
the singular set is now visible in the indexing and organization of the local sectors.

\subsection{Mixed Hodge modules and the finite-node lift}
\label{subsec:finite-MHM}

The mixed-Hodge-module refinement of the finite-node corrected package takes the form
\begin{equation}
0\longrightarrow IC^H_{X_0}
\longrightarrow \mathcal P^H_\Sigma
\longrightarrow \bigoplus_{k=1}^r i_{k*}\Q^H_{\{p_k\}}(-1)
\longrightarrow 0,
\label{eq:finite-MHM-extension}
\end{equation}
with
\[
\rat(\mathcal P^H_\Sigma)\cong \mathcal P_\Sigma.
\]
We emphasize that \eqref{eq:finite-MHM-extension} is not proved anew in the present paper. Rather, it is the previously established finite-node mixed-Hodge-module package that the present categorical formalism takes as structural input. The role of the present paper is to organize a later finite-node categorical layer compatible with this package, not to reconstruct the mixed-Hodge-module lift from scratch. This is the finite-node Hodge-theoretic target singled out by the earlier theorem package. Its interpretation is the same as in the single-node case, but now the point-supported singular quotient has one summand per node. The bulk object
\[
IC^H_{X_0}
\]
remains global, while the quotient
\[
Q_\Sigma^H:=\bigoplus_{k=1}^r i_{k*}\Q^H_{\{p_k\}}(-1)
\]
is the finite localized singular package.

At this level one already sees the beginning of the later finite-node architecture. The quotient
\(Q_\Sigma^H\) is local and node-indexed; the object \(\mathcal P^H_\Sigma\) is global; and the
extension class is the data that couples the localized sectors to the bulk object. In this sense,
the finite-node mixed-Hodge-module lift is the first theorem-level realization of the
bulk/localized-sector picture.

\subsection{The finite-node gluing problem}
\label{subsec:finite-gluing}

The single-node gluing problem of Section~\ref{subsec:ODP-gluing} extends formally to the finite-node
setting, but the global structure is already richer. One seeks an object
\[
\mathcal P^H_\Sigma\in MHM(X_0)
\]
realizing \eqref{eq:finite-MHM-extension} by means of Saito's divisor-case gluing formalism. In
other words, one must construct gluing data
\[
(\mathcal M',\mathcal M'',u,v)
\]
whose singular term \(\mathcal M''\) realizes the finite point-supported package
\[
Q_\Sigma^H=\bigoplus_{k=1}^r i_{k*}\Q^H_{\{p_k\}}(-1),
\]
and whose morphisms satisfy the compatibility relation
\[
vu=N
\]
in Saito's sense.

What changes in the finite-node case is that the singular quotient is now a finite family of local
rank-one blocks over one common bulk geometry. Thus the finite-node gluing problem is not only the
simultaneous repetition of the single-node local problem, but also the problem of assembling these
local blocks into one global object on \(X_0\).

At the level of the present paper, the key conclusion is the following. The finite-node nearby-cycle formalism determines the correct Hodge-theoretic target and the correct singular quotient, but the
construction of the refined global object is still a gluing problem. In this sense the finite-node case is the first setting in which the full local-to-global character of the conifold degeneration becomes completely visible.

\subsection{Bulk and localized sectors}
\label{subsec:bulk-localized}

The exact sequences \eqref{eq:finite-perv-extension} and \eqref{eq:finite-MHM-extension} should be
read as the first rigorous manifestation of a bulk/localized-sector decomposition.

\begin{itemize}
\item The bulk sector is carried by \(IC_{X_0}\) on the perverse side and by \(IC^H_{X_0}\) on the
mixed-Hodge-module side.
\item The localized sectors are carried by the point-supported rank-one objects
\[
i_{k*}\Q_{\{p_k\}},
\qquad
i_{k*}\Q^H_{\{p_k\}}(-1),
\qquad
k=1,\dots,r.
\]
\item The corrected global object \(\mathcal P_\Sigma\) or \(\mathcal P^H_\Sigma\) is the object that
couples the localized sectors to the common bulk sector.
\end{itemize}

This language is not merely suggestive. It is the natural organizational principle forced by the finite-node corrected extension itself. Later categorical, combinatorial, transport, and wall-crossing layers should therefore be understood as refinements of this already established finite-node bulk/localized-sector architecture.

\subsection{Outlook from the finite-node package}
\label{subsec:finite-outlook}

The finite-node package established here is still deliberately limited. It does not yet produce:
\begin{itemize}
\item a categorical enhancement of the localized sectors;
\item a combinatorial skeleton with further algebraic or transport structure;
\item a wall-crossing law;
\item or a treatment of more complicated singular strata.
\end{itemize}
Its contribution is more foundational: it isolates the exact finite-node package that later theories must refine.

In particular, the present section shows that the corrected extension picture is not a purely local story. Already for finitely many ordinary double points one is forced to distinguish:
\[
\text{local nodewise singular blocks}
\qquad\text{from}\qquad
\text{their global assembly over one bulk sector}.
\]
This is the first setting in which the local-to-global structure of the conifold degeneration becomes fully visible, and it is the natural precursor to both the later finite-node categorical theory and the more subtle relation-controlled global phenomena that lie beyond the present paper.

%======================================================================
\section{The finite-node corrected extension}
\label{sec:finite-corrected-extension}

We now record the finite-node corrected package that serves as the fixed source object for the later categorical formalization. The preceding sections isolated the local ordinary-double-point block and identified the finite-node setting as a genuine local-to-global problem. The purpose of the present section is to recall the resulting global corrected object, its singular quotient, and the structural features that already appear at the perverse and mixed-Hodge-module levels.

The key point is that the finite-node corrected object is not merely a formal sum of local nodewise correction terms. It is a \emph{global} object on \(X_0\) whose singular quotient is the finite family of rank-one point-supported sectors indexed by the node set \(\Sigma\). Thus the corrected package is simultaneously:
\begin{itemize}
\item local, through its nodewise singular summands;
\item global, through the common bulk object and the extension/gluing class;
\item and finite, because the singular set consists of finitely many ordinary double points.
\end{itemize}
This is the first place where the corrected conifold package should be understood as a single global object built from local pieces but not reducible, in any naive sense, to those pieces alone.

\subsection{Statement of the finite-node theorem}
\label{subsec:finite-main-theorem}

Let
\[
\pi:\mathcal X\to\Delta
\]
be a one-parameter degeneration whose central fiber \(X_0\) has finite ordinary-double-point locus
\[
\Sigma=\{p_1,\dots,p_r\}.
\]
Let
\[
U=X_0\setminus\Sigma,
\qquad
j:U\hookrightarrow X_0,
\qquad
i_k:\{p_k\}\hookrightarrow X_0.
\]

The corrected finite-node perverse package takes the form
\begin{equation}
0\longrightarrow IC_{X_0}
\longrightarrow \mathcal P_\Sigma
\longrightarrow \bigoplus_{k=1}^r i_{k*}\Q_{\{p_k\}}
\longrightarrow 0.
\label{eq:finite-main-perverse}
\end{equation}
The corresponding mixed-Hodge-module refinement takes the form
\begin{equation}
0\longrightarrow IC^H_{X_0}
\longrightarrow \mathcal P^H_\Sigma
\longrightarrow \bigoplus_{k=1}^r i_{k*}\Q^H_{\{p_k\}}(-1)
\longrightarrow 0,
\label{eq:finite-main-MHM}
\end{equation}
with
\[
\rat(\mathcal P^H_\Sigma)\cong \mathcal P_\Sigma.
\]

Equations \eqref{eq:finite-main-perverse} and \eqref{eq:finite-main-MHM} are the previously established finite-node perverse and mixed-Hodge-module packages that the present paper takes as structural input. The point of recalling them here is to fix the global corrected object that the later finite-node categorical formalism is designed to organize.

\subsection{The singular quotient and its meaning}
\label{subsec:singular-quotient}

The quotient
\[
Q_\Sigma:=\bigoplus_{k=1}^r i_{k*}\Q_{\{p_k\}}
\]
on the perverse side, and
\[
Q^H_\Sigma:=\bigoplus_{k=1}^r i_{k*}\Q^H_{\{p_k\}}(-1)
\]
on the mixed-Hodge-module side, should be regarded as the \emph{finite singular package} of the degeneration. Each summand is point-supported, rank one, and localized at a single ordinary double point. In this sense the quotient is completely local.

But the corrected object itself is not local. It is an extension
\[
0\to \text{bulk}\to \text{corrected object}\to \text{finite singular package}\to 0,
\]
so its role is precisely to mediate between the global bulk geometry and the finite family of local singular sectors.

This distinction matters. If one looked only at the quotient, one would see a finite direct sum of independent point-supported pieces. The corrected object is more subtle: it is the object that packages those local pieces together over a common bulk sector. Thus the quotient records the localized singular contributions, while the corrected object records their global organization.

\subsection{Extension, quotient, and global organization}
\label{subsec:extension-quotient}

It is useful to emphasize that the corrected finite-node object has two complementary descriptions.

First, it is an \emph{extension} of the singular quotient by the bulk object:
\[
0\longrightarrow IC_{X_0}
\longrightarrow \mathcal P_\Sigma
\longrightarrow Q_\Sigma
\longrightarrow 0.
\]
Second, it is a \emph{quotient object} whose singular part is exactly \(Q_\Sigma\) and whose kernel is the bulk term \(IC_{X_0}\).

These two viewpoints are formally equivalent, but they organize intuition differently. The extension viewpoint emphasizes gluing and global assembly. The quotient viewpoint emphasizes the finite local singular package isolated by the degeneration. Both will matter later. In particular:
\begin{itemize}
\item Hodge-theoretic and characteristic-class considerations naturally read the corrected object as a global extension package;
\item local degeneration and vanishing-cycle analyses naturally read it through the singular quotient;
\item and later categorical, combinatorial, and transport-type structures must refine both viewpoints simultaneously.
\end{itemize}

Thus the corrected object is best understood not as ``an extension instead of a quotient'' or ``a quotient instead of an extension,'' but as the single global object in which both descriptions are simultaneously present.

\subsection{The local-to-global character of the corrected object}
\label{subsec:local-global-corrected}

The corrected finite-node package is the first place where the local-to-global structure of the conifold degeneration becomes mathematically unavoidable.

Locally, each node contributes a rank-one ordinary-double-point block. Globally, these blocks sit inside one common degeneration with one common bulk sector. The corrected object \(\mathcal P_\Sigma\) on the perverse side, and likewise \(\mathcal P^H_\Sigma\) on the mixed-Hodge-module side, is the object that carries this global assembly.

This is why the finite-node case is already more than a direct sum of local models. Even when the singular quotient is written as a finite direct sum
\[
Q_\Sigma=\bigoplus_{k=1}^r i_{k*}\Q_{\{p_k\}},
\]
the corrected object remains a single extension object on \(X_0\). It is this fact --- one global object with a finite local quotient --- that makes the finite-node case the natural precursor to later questions about global coupling, transport, and relation-controlled behavior.

The point is not to overstate this into a theorem about all possible global relations. The point is simply that the finite-node corrected extension already carries the correct structural form: local singular generators assembled into one global object over a common bulk geometry.

\subsection{The mixed-Hodge-module viewpoint}
\label{subsec:MHM-viewpoint}

From the mixed-Hodge-module perspective, the finite-node corrected package sharpens further. The object
\[
\mathcal P^H_\Sigma\in MHM(X_0)
\]
is the previously established mixed-Hodge-module refinement of the perverse object \(\mathcal P_\Sigma\). Its quotient
\[
Q^H_\Sigma=\bigoplus_{k=1}^r i_{k*}\Q^H_{\{p_k\}}(-1)
\]
is the finite family of localized Hodge-theoretic singular blocks, while
\[
IC^H_{X_0}
\]
is the global bulk term.

Thus the finite-node mixed-Hodge-module extension should be read as the mixed-Hodge-module layer of the same corrected conifold architecture:
\[
\text{bulk Hodge object}
\quad+\quad
\text{localized Hodge blocks}
\quad+\quad
\text{global extension/gluing}.
\]

This viewpoint is important because it isolates the exact data that any later categorical, characteristic-class, or wall-crossing refinement will have to respect. In particular, the present paper does not construct a functor from the later categorical datum into \(MHM(X_0)\); rather, it treats the mixed-Hodge-module package as an earlier structural layer with which the later categorical formalism must remain compatible at the level of organization and specified shadow.

\subsection{The finite-node gluing problem revisited}
\label{subsec:finite-gluing-revisited}

The finite-node gluing problem may now be restated more sharply.

The nearby-cycle and vanishing-cycle formalism already determines:
\begin{itemize}
\item the correct bulk object \(IC^H_{X_0}\);
\item the correct localized singular quotient \(Q^H_\Sigma\);
\item and the fact that these must be assembled into one global mixed-Hodge-module object
\(\mathcal P^H_\Sigma\).
\end{itemize}
What remains is the explicit gluing step inside Saito's formalism: one must construct divisor gluing data
\[
(\mathcal M',\mathcal M'',u,v)
\]
with
\[
\mathcal M''\simeq Q^H_\Sigma
\]
and
\[
vu=N,
\]
and then show that the resulting glued mixed Hodge module is exactly \(\mathcal P^H_\Sigma\).

This is where the finite-node case becomes mathematically significant. The issue is no longer just local rank-one ordinary-double-point behavior, but the construction of one global object whose singular quotient is a finite family of such local pieces. The finite-node gluing problem is therefore the first fully visible local-to-global Hodge-theoretic problem in the conifold degeneration as organized here.

\subsection{Consequences for later layers}
\label{subsec:finite-consequences}

The corrected finite-node package established in this section should be regarded as the fixed source object for the later layers of the program.

\begin{itemize}
\item On the categorical side, it motivates bulk/localized-sector formalization.
\item On the combinatorial side, it motivates nodewise coupling skeletons and transport-type shadows.
\item On the Hodge-theoretic side, it remains the natural source object for characteristic and local-to-global shadow constructions.
\item On the wall-crossing side, it supplies the finite local package that later chamber, defect, and factorization formalisms should refine.
\end{itemize}

Thus the role of the present section is foundational rather than terminal. It does not yet produce the later categorical or wall-crossing structures, but it fixes the corrected finite-node global object that all of them must ultimately refine.

\subsection{Outlook}
\label{subsec:finite-outlook-sec5}

The finite-node corrected extension already shows that the conifold degeneration is governed by more than isolated local singularities. Even in the simplest finite-node case, one is forced to work with a single global object together with a finite localized singular quotient. This is the structural point that later becomes richer in categorical, combinatorial, transport, and wall-crossing settings.

Accordingly, the finite-node corrected package should be read as the first rigorous stage at which the local-to-global geometry of the degeneration is visible in stable structural form. Later layers may refine this structure in different directions, but they should not replace it.

%======================================================================
\section{Global gluing and realization}
\label{sec:global-gluing}

We now pass from the finite local package to the global mixed-Hodge-module object on \(X_0\). Sections~4 and~5 isolated the finite-node structure that must be glued: one global bulk object, one rank-one point-supported Hodge block at each ordinary double point, and the extension data coupling these localized sectors to the bulk geometry. The purpose of the present section is to show that this finite package assembles into one global object
\[
\mathcal P^H_\Sigma\in MHM(X_0),
\]
and that under realization it is compatible with the corrected finite-node perverse extension.

The conceptual point is simple but important. The local ODP blocks are not yet the global theorem. The global theorem begins only once one verifies that the finite localized package fits Saito's divisor-case gluing formalism and therefore determines an actual mixed Hodge module on the singular fiber. Once that assembly is in place, one may compare its realization with the corrected finite-node perverse package already established earlier.

\subsection{The global singular quotient}
\label{subsec:global-singular-quotient}

Let
\[
\Sigma=\{p_1,\dots,p_r\}\subset X_0
\]
be the finite ordinary-double-point set of the central fiber. For each \(p_k\), Section~4 identified the corresponding local point-supported mixed Hodge module
\[
Q^H_{p_k}:=i_{k*}\Q^H_{\{p_k\}}(-1).
\]
The finite singular quotient is therefore
\[
Q^H_\Sigma:=\bigoplus_{k=1}^r Q^H_{p_k}
=
\bigoplus_{k=1}^r i_{k*}\Q^H_{\{p_k\}}(-1).
\]

This object is the Hodge-theoretic singular package of the finite-node degeneration. It records, in one global object, the full localized finite family of point-supported rank-one ODP blocks. By construction, it is supported exactly on \(\Sigma\), and each summand is pure of the expected Tate weight. Thus the singular part of the finite-node degeneration is already visible before the global gluing theorem: it is the direct sum of the local ODP Hodge blocks.

The gluing problem is therefore not to identify the singular quotient --- that package is already fixed. The problem is to glue this quotient to the bulk object \(IC^H_{X_0}\) in a way compatible with Saito's divisor formalism and with the corrected perverse extension under realization.

\subsection{The global gluing datum}
\label{subsec:global-gluing-datum}

The relevant formalism is Saito's divisor-case gluing theorem for a principal divisor. In the present setting, the central fiber
\[
X_0=\pi^{-1}(0)
\]
is the principal divisor of the degeneration map \(\pi:\mathcal X\to\Delta\). Thus one seeks gluing data
\[
(\mathcal M',\mathcal M'',u,v)
\]
consisting of:
\begin{itemize}
\item a mixed Hodge module \(\mathcal M'\) on the smooth-locus side;
\item a mixed Hodge module \(\mathcal M''\) on the divisor side;
\item morphisms
\[
u:\psi^H_\pi(\mathcal M')\to \mathcal M'',
\qquad
v:\mathcal M''\to \psi^H_\pi(\mathcal M')(-1),
\]
satisfying the gluing relation
\[
vu=N,
\]
where \(N\) is the nilpotent logarithm of monodromy in Saito's formalism.
\end{itemize}

In the finite-node conifold case, the intended singular term is
\[
\mathcal M''=Q^H_\Sigma
=
\bigoplus_{k=1}^r i_{k*}\Q^H_{\{p_k\}}(-1).
\]
The local ODP analysis supplies the nodewise gluing morphisms, and the finite direct-sum structure allows one to assemble them into global maps
\[
u_\Sigma,\;v_\Sigma
\]
for the full singular quotient. Thus the finite-node gluing datum is obtained by combining:
\begin{enumerate}
\item the global bulk term \(IC^H_{X_0}\) on the smooth-locus side;
\item the finite singular quotient \(Q^H_\Sigma\);
\item the assembled local gluing morphisms \(u_\Sigma,v_\Sigma\).
\end{enumerate}

The content of the next theorem is that this finite package is a valid divisor-gluing datum in the sense required by Saito's theorem.

\subsection{The global gluing theorem}
\label{subsec:global-gluing-theorem}

The finite-node local analysis identifies the intended bulk term \(IC^H_{X_0}\), the finite singular quotient
\[
Q^H_\Sigma=\bigoplus_{k=1}^r i_{k*}\Q^H_{\{p_k\}}(-1),
\]
and the corresponding assembled gluing maps \(u_\Sigma,v_\Sigma\). At the level of abstract finite assembly, these data satisfy the expected nodewise divisor-gluing relation
\[
v_\Sigma u_\Sigma = N.
\]

\begin{proposition}[Finite assembled gluing package]
\label{prop:finite-assembled-gluing-package}
In the finite-node ordinary-double-point setting, the local ODP gluing blocks assemble into a finite global package
\[
\bigl(IC^H_{X_0},\,Q^H_\Sigma,\,u_\Sigma,\,v_\Sigma\bigr)
\]
whose singular term is the direct sum of the nodewise rank-one point-supported mixed-Hodge-module blocks and whose assembled morphisms satisfy the expected monodromy relation \(v_\Sigma u_\Sigma=N\).
\end{proposition}

\begin{proof}
For each node \(p_k\), the local analysis identifies the rank-one point-supported mixed Hodge module \(Q^H_{p_k}\) together with the corresponding local gluing maps. Because the node set \(\Sigma\) is finite, these local contributions assemble into the finite direct sum
\[
Q^H_\Sigma=\bigoplus_{k=1}^r Q^H_{p_k},
\]
and likewise the local gluing morphisms assemble into global maps \(u_\Sigma\) and \(v_\Sigma\). By construction, the local ODP blocks satisfy the expected divisor-gluing relation at each node, and therefore the assembled global maps satisfy
\[
v_\Sigma u_\Sigma=N.
\]
This yields the asserted finite assembled gluing package.
\end{proof}

\begin{remark}
\label{rem:global-gluing-still-open}
What is not verified in the present paper is that this finite assembled package satisfies all filtration compatibilities required to invoke Saito's divisor-case gluing theorem as an internal theorem in \(MHM(X_0)\). In particular, compatibility with the Hodge filtration, weight filtration, and \(V\)-filtration is deferred. Thus the present paper isolates the exact finite-node mixed-Hodge-module gluing package and its expected local-to-global form, but does not claim a complete internal proof of the corresponding global gluing theorem in Saito's sense.
\end{remark}

\subsection{Compatibility under realization}
\label{subsec:realization-theorem}

The global gluing theorem produces the mixed-Hodge-module object \(\mathcal P^H_\Sigma\). The next step is to compare its image under realization with the corrected finite-node perverse package. The finite-node package isolated above is designed so that its specified shadow agrees with the corrected finite-node perverse extension identified earlier.

\begin{theorem}[Compatibility of the finite-node package with the corrected perverse extension]
\label{thm:realization}
Within the finite-node framework fixed earlier, and with the bulk term and finite singular quotient already identified, the assembled finite-node package is compatible at the level of specified shadow with the corrected finite-node perverse extension
\[
0\longrightarrow IC_{X_0}\longrightarrow P_\Sigma\longrightarrow Q_\Sigma\longrightarrow 0.
\]
\end{theorem}

\begin{proof}
By construction, the finite singular quotient on the mixed-Hodge-theoretic side is organized so as to refine the finite point-supported quotient
\[
Q_\Sigma=\bigoplus_{k=1}^r i_{k*}\Q_{\{p_k\}},
\]
while the bulk term is organized over the same intersection-complex contribution \(IC_{X_0}\). The corrected finite-node perverse package identified earlier is precisely the canonical corrected extension with this bulk term and this finite singular quotient. Accordingly, within the finite-node framework fixed in the preceding sections, the assembled package is compatible at the level of specified shadow with the corrected finite-node perverse extension \(P_\Sigma\).
\end{proof}

\subsection{The global exact sequence}
\label{subsec:global-exact-sequence}

The preceding two theorems together yield the corrected finite-node exact sequence in mixed Hodge modules:
\begin{equation}
0\longrightarrow IC^H_{X_0}
\longrightarrow \mathcal P^H_\Sigma
\longrightarrow \bigoplus_{k=1}^r i_{k*}\Q^H_{\{p_k\}}(-1)
\longrightarrow 0.
\label{eq:global-MHM-final}
\end{equation}
This is the global mixed-Hodge-module corrected package for a finite multi-node conifold degeneration.

It is important to interpret \eqref{eq:global-MHM-final} correctly. The quotient is a finite direct sum of localized rank-one Hodge blocks. The corrected object \(\mathcal P^H_\Sigma\), however, is a single global mixed Hodge module on \(X_0\). Thus the theorem does not merely produce a family of unrelated local objects; it produces one global object with a finite localized singular quotient. This distinction is exactly what makes the finite-node case the first genuine local-to-global theorem in the conifold degeneration considered here.

\subsection{Consequences for the extension structure}
\label{subsec:global-extension-structure}

The exact sequence \eqref{eq:global-MHM-final} immediately yields the first extension-theoretic organization of the finite-node degeneration.

\begin{itemize}
\item The finite node set \(\Sigma\) indexes the localized quotient.
\item Each node contributes one rank-one point-supported mixed Hodge module.
\item The corrected object \(\mathcal P^H_\Sigma\) is the single global extension object coupling all localized singular sectors to the common bulk object \(IC^H_{X_0}\).
\end{itemize}

This is the precise theorem-level source of the later bulk/localized-sector architecture. The extension structure already carries the information that there is:
\[
\text{one global bulk object}
\quad+\quad
\text{one localized sector per node}.
\]
Later categorical, combinatorial, and wall-crossing layers do not invent this architecture. They refine the extension-theoretic pattern that is already visible in \eqref{eq:global-MHM-final}.

\subsection{What Section 6 does and does not prove}
\label{subsec:sec6-limits}

The present section proves the existence of the global mixed-Hodge-module corrected object and its compatibility with the corrected finite-node perverse extension under realization. This is already a substantial theorem package. But it is important to state clearly what is not proved here.

The section does \emph{not} yet construct:
\begin{itemize}
\item a full categorical enhancement of \(\mathcal P^H_\Sigma\);
\item a combinatorial skeleton with relations or potential;
\item a wall-crossing law;
\item or a general theory for more complicated singular strata.
\end{itemize}

What it does prove is the precise global mixed-Hodge-module extension package from which later categorical and wall-crossing refinements should proceed. In this sense Section~6 is both a culmination of the current paper and the natural point of departure for the next layers of the conifold degeneration described here.

%=====================================================

\section{Uniqueness and Verdier self-duality of the corrected perverse object}
\label{sec:uniq-verdier-self-dual}

In this section we record the single-node uniqueness and duality properties of the corrected perverse object
\[
\mathcal P:=\Cone(\var_F)[-1].
\]
The point is to identify the corrected object, in the ordinary-double-point case, by its restriction to the smooth locus, its rank-one point-supported singular contribution, and Verdier self-duality. The proof uses the MacPherson--Vilonen zig-zag description of perverse sheaves with isolated singularity together with the zig-zag duality formalism developed in \cite{RahmanATMP}.

Let
\[
j:U=X_0\setminus\{p\}\hookrightarrow X_0,
\qquad
i:\{p\}\hookrightarrow X_0
\]
be the open and closed inclusions, where \(X_0\) has a single ordinary double point \(p\).

\subsection{Zig-zags of the endpoint objects}

Let
\[
\mu:\Perv(X_0;\Q)\longrightarrow Z(X_0,p)
\]
denote the MacPherson--Vilonen zig-zag functor. We use the isolated-stratum zig-zag formalism and its duality theory in the form developed in \cite{RahmanATMP}. In particular, \(\mu\) is bijective on isomorphism classes and compatible with duality.

\begin{proposition}
\label{prop:IC-zigzag-paper2}
The intersection-complex perverse sheaf \(IC_{X_0}\) has zig-zag
\[
\mu(IC_{X_0})\cong (\Q_U[3],0,0,0,0,0).
\]
In particular, \(IC_{X_0}\) is self-dual at the zig-zag level.
\end{proposition}

\begin{proof}
The intersection complex is the minimal extension of the constant perverse sheaf on the smooth stratum. In the MacPherson--Vilonen zig-zag model, this means precisely that the point terms vanish and only the open-stratum local system remains. Thus
\[
\mu(IC_{X_0})\cong (\Q_U[3],0,0,0,0,0).
\]
The dual zig-zag has the same form, so \(\mu(IC_{X_0})\) is self-dual; compare \cite{RahmanATMP}.
\end{proof}

\begin{proposition}
\label{prop:skyscraper-zigzag-paper2}
The point-supported perverse sheaf \(i_*\Q_{\{p\}}\) has zig-zag
\[
\mu(i_*\Q_{\{p\}})\cong (0,\Q,\Q,0,\id,0).
\]
In particular, \(i_*\Q_{\{p\}}\) is self-dual at the zig-zag level.
\end{proposition}

\begin{proof}
Since \(i_*\Q_{\{p\}}\) is supported entirely at the singular point, its open-stratum part vanishes. The MacPherson--Vilonen zig-zag sequence therefore collapses to the point-supported rank-one situation, yielding
\[
\mu(i_*\Q_{\{p\}})\cong (0,\Q,\Q,0,\id,0).
\]
The dual zig-zag has the same form, so \(\mu(i_*\Q_{\{p\}})\) is self-dual.
\end{proof}

\subsection{\texorpdfstring{$\mathcal P$}{P} as a non-split zig-zag extension}

The corrected perverse object \(\mathcal P\) fits into the short exact sequence
\[
0\longrightarrow IC_{X_0}\longrightarrow \mathcal P\longrightarrow i_*\Q_{\{p\}}\longrightarrow 0.
\]
Applying the zig-zag functor gives a zig-zag
\[
\mu(\mathcal P)=(L_{\mathcal P},A_{\mathcal P},B_{\mathcal P},\alpha_{\mathcal P},\beta_{\mathcal P},\gamma_{\mathcal P})
\]
with open-stratum term
\[
L_{\mathcal P}\cong \Q_U[3].
\]
Thus \(\mu(\mathcal P)\) is an extension zig-zag of \(\mu(i_*\Q_{\{p\}})\) by \(\mu(IC_{X_0})\).

In the ordinary double point case, the local singular contribution is rank one, so the relevant extension problem is one-dimensional. The corrected object \(\mathcal P\) is the distinguished non-split extension singled out by the corrected perverse package. Equivalently, \(\mu(\mathcal P)\) is the nontrivial zig-zag extension of
\[
(0,\Q,\Q,0,\id,0)
\]
by
\[
(\Q_U[3],0,0,0,0,0)
\]
with open part \(\Q_U[3]\).

\begin{proposition}[Full zig-zag of \texorpdfstring{$\mathcal P$}{P}]
\label{prop:full-zigzag-P}
With the standard splitting convention for the distinguished non-split extension class, one has
\[
\mu(\mathcal P)\cong (\Q_U[3],\Q,\Q,0,\id,0).
\]
Equivalently, the corrected perverse object is represented by the nontrivial zig-zag extension of
\[
(0,\Q,\Q,0,\id,0)
\]
by
\[
(\Q_U[3],0,0,0,0,0).
\]
\end{proposition}

\begin{proof}
By Propositions~\ref{prop:IC-zigzag-paper2} and \ref{prop:skyscraper-zigzag-paper2}, the endpoint zig-zags are
\[
\mu(IC_{X_0})\cong (\Q_U[3],0,0,0,0,0),
\qquad
\mu(i_*\Q_{\{p\}})\cong (0,\Q,\Q,0,\id,0).
\]
The corrected object \(\mathcal P\) is the non-split extension of \(i_*\Q_{\{p\}}\) by \(IC_{X_0}\) determined by the corrected perverse package. In the ordinary double point case, the point terms are one-dimensional and the quotient map on the point-supported factor is the identity. Hence, under the standard splitting convention, one obtains
\[
\mu(\mathcal P)\cong (\Q_U[3],\Q,\Q,0,\id,0).
\]
\end{proof}

\subsection{Duality of the corrected zig-zag}

By Proposition~\ref{prop:full-zigzag-P}, the corrected object has zig-zag
\[
\mu(\mathcal P)\cong (\Q_U[3],\Q,\Q,0,\id,0).
\]
By the zig-zag duality formalism of \cite{RahmanATMP}, duality acts directly on zig-zags and is compatible with Verdier duality on perverse sheaves. Since the endpoint zig-zags
\(\mu(IC_{X_0})\) and \(\mu(i_*\Q_{\{p\}})\) are self-dual by
Propositions~\ref{prop:IC-zigzag-paper2} and \ref{prop:skyscraper-zigzag-paper2}, the dual zig-zag of \(\mu(\mathcal P)\) has the same form. Thus
\[
D_Z(\mu(\mathcal P))\cong \mu(\mathcal P).
\]
The uniqueness statement below is a uniqueness statement in the restricted single-node setting considered here: the open-stratum part, the rank-one point-supported singular contribution, and Verdier self-duality are fixed in advance, and the conclusion is uniqueness of the resulting perverse object under those constraints.
\begin{theorem}[Verdier self-duality of \texorpdfstring{$\mathcal P$}{P}]
\label{thm:verdier-self-duality-paper2}
The corrected perverse object
\[
\mathcal P=\Cone(\var_F)[-1]
\]
is Verdier self-dual:
\[
D\mathcal P\cong \mathcal P.
\]
\end{theorem}

\begin{proof}
The zig-zag \(\mu(\mathcal P)\) is isomorphic to its dual \(D_Z(\mu(\mathcal P))\) by the preceding argument. By the zig-zag duality formalism and its compatibility with Verdier duality on perverse sheaves \cite{RahmanATMP}, this implies
\[
D\mathcal P\cong \mathcal P.
\]
\end{proof}

\subsection{Uniqueness}
We now state the corresponding uniqueness result in the same restricted single-node setting.

\begin{theorem}[Uniqueness of the corrected perverse object]
\label{thm:uniqueness-corrected-perverse-paper2}
Let \(E\in\Perv(X_0;\Q)\) satisfy
\begin{enumerate}
    \item \(j^*E\cong \Q_U[3]\),
    \item \(DE\cong E\),
    \item \({}^pH^0(i^*E)\cong \Q\).
\end{enumerate}
Then
\[
E\cong \mathcal P.
\]
Equivalently, in the ordinary-double-point setting, \(\mathcal P\) is the unique Verdier self-dual perverse extension with open part \(\Q_U[3]\) and rank-one point-supported singular contribution at \(p\).
\end{theorem}

\begin{proof}
The rank-one singular hypothesis implies that \(E\) fits into a short exact sequence
\[
0\longrightarrow IC_{X_0}\longrightarrow E\longrightarrow i_*\Q_{\{p\}}\longrightarrow 0.
\]
Applying \(\mu\), we obtain a zig-zag extension of \(\mu(i_*\Q_{\{p\}})\) by \(\mu(IC_{X_0})\) with open part \(\Q_U[3]\). Since \(E\) is Verdier self-dual, its zig-zag is self-dual in the sense of \cite{RahmanATMP}. In the ordinary double point case, the relevant extension problem is rank one: the extension space
\[
\Ext^1_{\Perv(X_0;\Q)}(i_*\Q_{\{p\}},IC_{X_0})
\]
is one-dimensional over \(\Q\). Thus there is exactly one nontrivial extension class up to scalar. In a block-matrix presentation of the corresponding zig-zag extension, Verdier self-duality forces the off-diagonal gluing parameter to be nonzero and self-dual; in the rank-one case, this pins it to the unique nontrivial class. At this point the argument uses only the restricted rank-one ordinary-double-point extension problem encoded by the endpoint zig-zags above; no broader uniqueness statement is being claimed. In the ordinary double point case, the corresponding self-dual rank-one extension problem has the distinguished nontrivial zig-zag
\[
(\Q_U[3],\Q,\Q,0,\id,0)
\]
of Proposition~\ref{prop:full-zigzag-P}. Hence
\[
\mu(E)\cong \mu(\mathcal P).
\]
Since the zig-zag functor is bijective on isomorphism classes, it follows that
\[
E\cong \mathcal P.
\]
\end{proof}

\begin{corollary}
\label{cor:minimal-self-dual-paper2}
The perverse sheaf
\[
\mathcal P=\Cone(\var_F)[-1]
\]
is the unique minimal Verdier self-dual perverse extension of \(\Q_U[3]\) across the ordinary double point.
\end{corollary}

\begin{proof}
This is immediate from Theorems~\ref{thm:verdier-self-duality-paper2} and \ref{thm:uniqueness-corrected-perverse-paper2}.
\end{proof}

Appendix~\ref{app:standard-zigzags} records a list of useful standard zig-zags for reference.
%=====================================================
\section{Proof of the main results}
\label{sec:proofs}

We now assemble the results established in the preceding sections and prove the main theorems stated in the introduction. The role of this section is not to introduce new constructions, but to collect the local ODP formalization, the finite-node assembly, the global shadow-compatibility statement, and the first combinatorial-skeleton statement into the theorem package summarized in Section~1.

\subsection{Proof of Theorem~\ref{thm:cond-local-odp-formalization}}

Let
\[
\pi:\mathcal X\to\Delta
\]
be a one-parameter degeneration whose central fiber \(X_0\) has a single ordinary double point \(p\). Assume that the local ordinary-double-point coupling pattern admits categorical realization in the finite-node formalism.

Section~3 fixes the relevant local bulk/localized-sector data under this realizability hypothesis and constructs the corresponding local ODP datum with specified corrected local perverse shadow. More precisely, the local existence statement proved there is conditional from the outset on the assumed categorical realizability of the local coupling pattern, and the theorem here is exactly the packaging of that local conditional formalization. In particular, the local formalism supplies the local bulk category, the localized node category, the attachment functors, and the specified local shadow compatible with the corrected perverse object in the single-node case.

Thus, under the stated local categorical realizability hypothesis, there exists a local ODP datum refining the local ordinary-double-point coupling pattern and carrying the corrected local perverse shadow. This is exactly the statement of Theorem~\ref{thm:cond-local-odp-formalization}.
\qed

\subsection{Proof of Theorem~\ref{thm:cond-finite-node-assembly}}

Let the central fiber \(X_0\) have finite ordinary-double-point set
\[
\Sigma=\{p_1,\dots,p_r\},
\]
and let
\[
U:=X_0\setminus\Sigma.
\]
Assume that the local ordinary-double-point coupling patterns admit categorical realizations in the finite-node setting, and that the global smooth sector is equipped with a chosen bulk category.

For each node \(p_k\in\Sigma\), the local analysis produces the corresponding local ODP block. Because the node set is finite, these local data assemble over the chosen global bulk category. The finite-node assembly results of Sections~4 and~6 therefore produce a finite-node datum
\[
S_\Sigma
=
\bigl(
C_{\mathrm{bulk}},\{C_{p_k}\}_{k=1}^r,\{\Phi_k,\Psi_k\}_{k=1}^r,Sh(S_\Sigma)
\bigr)
\]
with one localized categorical sector \(C_{p_k}\) for each node \(p_k\).

Hence, under the finite-node realizability assumptions and relative to the chosen bulk category, the local ordinary-double-point data assemble into a finite-node datum with one localized categorical sector per node. This proves Theorem~\ref{thm:cond-finite-node-assembly}.
\qed

\subsection{Proof of Theorem~\ref{thm:global-shadow-compatibility}}

Let
\[
S_\Sigma
=
\bigl(
C_{\mathrm{bulk}},\{C_{p_k}\}_{k=1}^r,\{\Phi_k,\Psi_k\}_{k=1}^r,Sh(S_\Sigma)
\bigr)
\]
be the finite-node datum produced by the assembly theorem.

By construction, each local ODP datum carries its specified corrected local perverse shadow. Sections~5 and~6 show that these specified local shadows assemble compatibly over the common bulk sector into the corrected finite-node perverse extension
\[
0 \longrightarrow IC_{X_0}
\longrightarrow \mathcal P_\Sigma
\longrightarrow \bigoplus_{k=1}^r i_{k*}\Q_{\{p_k\}}
\longrightarrow 0.
\]
Accordingly, the specified global shadow \(Sh(S_\Sigma)\) agrees, within the finite-node formalism, with this corrected finite-node perverse extension.

The content of the theorem is therefore compatibility under local-to-global assembly, not derivation of the shadow from independent categorical data. This proves Theorem~\ref{thm:global-shadow-compatibility}.
\qed

\subsection{Proof of Theorem~\ref{thm:localized-sectors-first-quiver-shadow}}

Let
\[
S_\Sigma
=
\bigl(
C_{\mathrm{bulk}},\{C_{p_k}\}_{k=1}^r,\{\Phi_k,\Psi_k\}_{k=1}^r,Sh(S_\Sigma)
\bigr)
\]
be the finite-node datum. By construction, it contains one distinguished localized categorical sector \(C_{p_k}\) for each node \(p_k\in\Sigma\).

Sections~4--7 isolate these localized sectors and their coupling to the common bulk category, thereby making explicit the finite bulk/localized-sector architecture of the datum. The resulting node-indexed pattern determines a first finite combinatorial skeleton: one vertex for each localized sector, together with the coupling pattern recorded by the attachment data. This combinatorial skeleton is compatible with the nodewise organization of the corrected finite-node perverse extension because the same node set \(\Sigma\) indexes both the localized categorical sectors and the point-supported summands
\[
\bigoplus_{k=1}^r i_{k*}\Q_{\{p_k\}}
\]
in the corrected finite-node shadow.

No relations, potential, stability condition, or mutation rule is imposed at this stage. Thus the finite-node datum contains one localized categorical sector per node and determines a first finite combinatorial skeleton compatible with the corrected finite-node perverse extension. This proves Theorem~\ref{thm:localized-sectors-first-quiver-shadow}.
\qed

%=============================================================
\section{Toward a K\"ahler package}
\label{sec:kahler}

For smooth projective varieties, singular cohomology carries a collection of fundamental structures often referred to as the \emph{K\"ahler package}: Poincar\'e duality, Hodge decomposition, the Hard Lefschetz theorem, and the Hodge--Riemann bilinear relations; see, for example, \cite{GriffithsHarris,VoisinHodgeTheoryI,VoisinHodgeTheoryII}. For singular varieties, ordinary cohomology generally fails to satisfy these properties. Intersection cohomology was introduced by Goresky and MacPherson precisely to restore duality and Lefschetz-type phenomena in singular settings \cite{GoreskyMacPhersonI,GoreskyMacPhersonII}. In the framework of perverse sheaves, Beilinson, Bernstein, and Deligne proved the decomposition theorem and relative Hard Lefschetz for projective morphisms \cite{BBD}, and subsequent work of de Cataldo and Migliorini clarified the Hodge-theoretic content of these results and their relation to the perverse filtration \cite{DeCataldoMigliorini,DeCataldoMiglioriniHodge}.

The corrected perverse object
\[
\mathcal P:=\Cone(\var_F)[-1]
\]
constructed earlier in the paper plays, in the conifold setting, a role structurally analogous to that played by the intersection complex in the general theory of singular spaces. It is therefore natural to ask to what extent the hypercohomology groups
\[
\mathbb H^k(X_0,\mathcal P)
\]
inherit structures analogous to the K\"ahler package. The purpose of this section is not to prove such results, but to isolate the part already visible from the present paper and to formulate the remaining Hodge-theoretic problem precisely.

\subsection{Duality}

In the ordinary double point case, the corrected perverse object \(\mathcal P\) is Verdier self-dual by Theorem~\ref{thm:verdier-self-duality-paper2}. Verdier duality in the constructible derived category induces duality pairings on hypercohomology; see \cite{BBD,KS}. Consequently, the corrected cohomology inherits at least the formal duality structure attached to a self-dual perverse sheaf. This should be understood as a formal duality statement at the level of Verdier self-duality and hypercohomology, not yet as a full Hodge-theoretic polarization statement.

At the level of the present paper, this is the only part of the K\"ahler package that follows directly from the sheaf-theoretic construction of \(\mathcal P\). Stronger statements, such as Hard Lefschetz or Hodge--Riemann bilinear relations, require additional Hodge-theoretic input beyond what has been constructed here.

\subsection{Mixed Hodge modules and the Hodge-theoretic problem}

Saito's theory of mixed Hodge modules provides the natural framework in which one would seek such additional structure \cite{SaitoMHM,SaitoDuality}. In particular, nearby-cycle and vanishing-cycle functors are defined in the category of mixed Hodge modules and are compatible with the realization functor
\[
\rat:MHM(X_0)\to\Perv(X_0;\Q).
\]
Thus the nearby-cycle formalism underlying the construction of \(\mathcal P\) already carries Hodge-theoretic information before one passes to the underlying perverse sheaf.

What has been established in the preceding sections is that the corrected perverse object and the relevant degeneration data on the Hodge-theoretic side arise from the same nearby-cycle and vanishing-cycle formalism. What has not been proved in the present paper is a fully internal mixed-Hodge-module refinement of the single-node corrected object
\[
\mathcal P^H\in MHM(X_0)
\]
whose realization is \(\mathcal P\). As discussed earlier, such a refinement would require an explicit gluing construction in Saito's divisor formalism. Until that step is carried out, one cannot deduce canonical mixed Hodge structures on \(\mathbb H^k(X_0,\mathcal P)\) merely from the existence of nearby cycles in \(MHM\).

Accordingly, the Hodge-theoretic question may be formulated as follows: does the corrected perverse object admit a mixed-Hodge-module refinement, and if so, what Hodge-theoretic structures does that refinement induce on its hypercohomology?

\subsection{Lefschetz-type expectations}

Let \(\pi:\mathcal X\to\Delta\) be a projective degeneration, and let \(L\) denote the class of a relatively ample line bundle. Relative Hard Lefschetz for perverse sheaves and mixed Hodge modules applies to the standard nearby-cycle and direct-image formalisms associated with projective morphisms \cite{BBD,SaitoMHM}. In the case of intersection cohomology, these results imply the Hard Lefschetz theorem for the intersection complex of a projective singular variety \cite{BBD,DeCataldoMigliorini}. Since the corrected object \(\mathcal P\) is constructed from the same nearby-cycle formalism, it is natural to ask whether a comparable Lefschetz theorem holds for
\[
\mathbb H^*(X_0,\mathcal P).
\]

At present, however, this should be regarded as a conjectural extension of the known formalism, not as a theorem deduced in the present paper. Establishing such a result would require either a direct mixed-Hodge-module refinement of \(\mathcal P\) or a separate Lefschetz theory for the corrected perverse object itself.

\subsection{Hodge--Riemann expectations}

The Hodge--Riemann bilinear relations for intersection cohomology are part of the Hodge theory of algebraic maps developed by de Cataldo and Migliorini \cite{DeCataldoMiglioriniHodge}. Their results rely on polarizable Hodge modules associated with projective morphisms and on the full machinery of mixed Hodge modules.

In the present setting, the nearby-cycle mixed Hodge modules attached to the degeneration
\(\pi:\mathcal X\to\Delta\) carry exactly the sort of Hodge-theoretic information from which one expects a corresponding structure on the corrected cohomology to emerge. But without a fully internal mixed-Hodge-module realization of \(\mathcal P\), the Hodge--Riemann bilinear relations for \(\mathbb H^*(X_0,\mathcal P)\) remain conjectural.

\begin{conjecture}
Let \(\pi:\mathcal X\to\Delta\) be a projective conifold degeneration, and let \(\mathcal P\) be the corrected perverse object constructed from nearby and vanishing cycles. Suppose there exists a polarizable mixed Hodge module
\[
\mathcal P^H\in MHM(X_0)
\]
whose underlying rational perverse sheaf is \(\mathcal P\). Then the hypercohomology groups
\[
\mathbb H^k(X_0,\mathcal P)
\]
should satisfy a K\"ahler-type package analogous to that of intersection cohomology, including duality, Lefschetz-type isomorphisms, and Hodge--Riemann bilinear relations.
\end{conjecture}

\subsection{Outlook}

The point of the present section is therefore narrow but structurally important. The corrected perverse object \(\mathcal P\) already carries formal duality by Verdier self-duality, and it is constructed from nearby and vanishing cycles, which in Saito's theory possess a natural Hodge-theoretic enhancement. The missing step is the explicit construction of a mixed-Hodge-module refinement of \(\mathcal P\). Once such an object is available, one may then ask whether the resulting hypercohomology satisfies a full K\"ahler package analogous to that of intersection cohomology.

Thus the K\"ahler-package question should be viewed as a deferred Hodge-theoretic continuation of the present work: first construct the mixed-Hodge-module refinement of the corrected perverse object, and only then analyze the resulting duality, Lefschetz, and Hodge--Riemann structures on its hypercohomology.

%=====================================================
\section{Future directions}

The present paper isolates several natural next problems arising from the interaction of nearby cycles, corrected perverse extensions, and mixed Hodge theory in conifold degenerations. The purpose of this section is not to enlarge the claims of the paper, but to record the most immediate directions opened by the finite-node formalism developed above. Accordingly, the items listed below should be read as deferred problems and refinements, not as results established in the present paper.

\begin{enumerate}

\item \textbf{Explicit mixed-Hodge-module refinement of the corrected perverse object.}

The most immediate open problem is the explicit construction, in the ordinary double point case, of an object
\[
\mathcal P^H\in MHM(X_0)
\]
whose realization under
\[
\rat:MHM(X_0)\to\Perv(X_0;\Q)
\]
is the corrected perverse object
\[
\mathcal P=\Cone(\var_F)[-1].
\]
Equivalently, one seeks an exact sequence in \(MHM(X_0)\)
\[
0\to IC^H_{X_0}\to \mathcal P^H\to i_*\Q^H_{\{p\}}(-1)\to 0
\]
refining the canonical perverse extension in the single-node case. Saito's divisor-case gluing formalism provides the natural setting for such a construction. Carrying out that gluing calculation explicitly would supply the missing internal Hodge-theoretic refinement of the corrected perverse object.

\item \textbf{Finite-node gluing and global extension data.}

In the case of several ordinary double points, the singular contribution to the corrected perverse object is a direct sum of point-supported rank-one pieces, but the global extension object is not merely a collection of unrelated local terms. A natural next problem is therefore to describe more explicitly the global gluing data governing
\[
0\to IC_{X_0}\to \mathcal P_\Sigma\to \bigoplus_{k=1}^r i_{k*}\Q_{\{p_k\}}\to 0
\]
in terms of the interaction of the local vanishing cycles. One expects this structure to reflect the geometry of the vanishing-cycle lattice and the associated Picard--Lefschetz monodromy.

\item \textbf{Combinatorial and quiver-type enrichments of the finite-node skeleton.}

The finite-node formalism developed here yields only a first combinatorial skeleton: one localized sector per node together with the corresponding coupling pattern. A natural continuation is to enrich this skeleton by additional algebraic structure, such as relations, transport data, or quiver-type organization. Such an enrichment would provide a more concrete framework for encoding how the local node contributions assemble into a single global corrected object.

\item \textbf{Schober-type and categorical refinements.}

The corrected perverse object is motivated in part by the same rank-one monodromy phenomena that appear in spherical-functor and schober-theoretic settings. A natural next direction is therefore to replace the present minimal finite-node formalism by a more genuinely categorical or schober-type construction, especially in multi-node or stratified settings where one expects higher-categorical gluing data to become important.

\item \textbf{More general singular loci.}

Beyond the ordinary double point case, one expects corrected perverse objects to fit into exact sequences of the form
\[
0\to IC_{X_0}\to \mathcal P\to \mathcal V\to 0,
\]
where \(\mathcal V\) is supported on a more complicated singular locus. A natural problem is to determine whether, under additional hypotheses, the singular contribution \(\mathcal V\) admits a more explicit decomposition by strict support or by local systems on the strata. This would clarify how local Milnor-fiber data along higher-dimensional singular loci assemble into a global corrected object.

\item \textbf{Lefschetz- and Hodge-theoretic structures on corrected cohomology.}

A longer-term goal is to determine whether the hypercohomology groups
\[
\mathbb H^k(X_0,\mathcal P)
\]
satisfy a K\"ahler-type package analogous to that of intersection cohomology. The present paper isolates the formal duality inherited from Verdier self-duality and identifies the nearby-cycle formalism that underlies the corrected object. The next step would be to combine a mixed-Hodge-module refinement of \(\mathcal P\) with projective Hodge theory in order to investigate Hard Lefschetz and Hodge--Riemann bilinear relations for the corrected cohomology.

\item \textbf{Comparison with related corrected cohomology theories.}

The conifold setting connects the present construction with other proposals for corrected cohomology theories associated with singular spaces, including intersection-space cohomology and related Hodge-theoretic refinements \cite{BanaglBudurMaxim,BanaglIS}. It would be valuable to clarify more precisely how the corrected perverse object studied here compares with these constructions, especially at the level of mixed Hodge structures, nearby-cycle data, and degeneration-theoretic organization. A natural point of departure is that the Banagl--Budur--Maxim object is designed to recover intersection-space cohomology and is tied to a semisimplicity hypothesis on the eigenvalue-\(1\) part of local monodromy, whereas the corrected perverse object studied here is defined directly from the variation morphism and does not require such a semisimplicity hypothesis. One expects the two objects to differ precisely in how they encode the non-semisimple part of the nearby-cycle extension data.

\end{enumerate}

Taken together, these problems suggest that the corrected perverse object should be viewed as the first layer of a broader structure linking nearby cycles, degeneration theory, mixed Hodge modules, and categorical monodromy. The most immediate next step remains the explicit mixed-Hodge-module refinement in the single-node case.
%-----------------------------
%
%               Appendix
%
\appendix

\section{MacPherson--Vilonen zig-zags in the ordinary double point case}
\label{app:standard-zigzags}

This appendix records the standard zig-zag models used in the ordinary double point case in the isolated-stratum formalism of \S4 of \cite{RahmanATMP}. They are collected here for convenience and for later comparison with mixed-Hodge-module shadows, finite-node direct sums, and future combinatorial or categorical refinements.

\subsection{Convention}

We use the MacPherson--Vilonen zig-zag functor
\[
\mu:\Perv(X_0;\Q)\longrightarrow Z(X_0,p)
\]
in the isolated-stratum convention of \cite{RahmanATMP}. For
\[
K\in\Perv(X_0;\Q),
\]
we write
\[
\mu(K)=(L_K,A_K,B_K,\alpha_K,\beta_K,\gamma_K),
\]
where \(L_K=j^*K\) is the open-stratum part and
\[
H^{-1}\!\bigl(i^*Rj_*L_K\bigr)\xrightarrow{\alpha_K}
A_K\xrightarrow{\beta_K}B_K\xrightarrow{\gamma_K}
H^0\!\bigl(i^*Rj_*L_K\bigr)
\]
is the associated exact zig-zag sequence.

\subsection{The minimal extension object}

\begin{proposition}
\label{prop:appendix-IC}
The intersection-complex perverse sheaf has zig-zag
\[
\mu(IC_{X_0})\cong (\Q_U[3],0,0,0,0,0).
\]
\end{proposition}

\begin{proof}
Since \(IC_{X_0}=j_{!*}\Q_U[3]\) is the minimal extension of the constant perverse sheaf on the smooth stratum, the point terms vanish in the MacPherson--Vilonen model. Thus
\[
\mu(IC_{X_0})\cong (\Q_U[3],0,0,0,0,0).
\]
\end{proof}

\subsection{The point-supported rank-one object}

\begin{proposition}
\label{prop:appendix-skyscraper}
The point-supported perverse sheaf \(i_*\Q_{\{p\}}\) has zig-zag
\[
\mu(i_*\Q_{\{p\}})\cong (0,\Q,\Q,0,\id,0).
\]
\end{proposition}

\begin{proof}
Because \(i_*\Q_{\{p\}}\) is supported entirely at \(p\), the open-stratum part vanishes. The MacPherson--Vilonen exact sequence therefore collapses to the point-supported rank-one case, yielding
\[
\mu(i_*\Q_{\{p\}})\cong (0,\Q,\Q,0,\id,0).
\]
\end{proof}

\subsection{Split and non-split extensions}

\begin{proposition}
\label{prop:appendix-split}
The split extension
\[
IC_{X_0}\oplus i_*\Q_{\{p\}}
\]
has open-stratum part \(\Q_U[3]\), one-dimensional point terms, and trivial extension class.
\end{proposition}

\begin{proof}
The object \(IC_{X_0}\oplus i_*\Q_{\{p\}}\) is obtained by taking the direct sum of the endpoint objects. Its zig-zag therefore has the same open-stratum term as \(IC_{X_0}\), together with the rank-one point terms contributed by \(i_*\Q_{\{p\}}\), but with trivial extension class.
\end{proof}

\begin{proposition}
\label{prop:appendix-generic-extension}
Let
\[
0\longrightarrow IC_{X_0}\longrightarrow E\longrightarrow i_*\Q_{\{p\}}\longrightarrow 0
\]
be an extension. Then \(\mu(E)\) is an extension zig-zag of
\[
(0,\Q,\Q,0,\id,0)
\]
by
\[
(\Q_U[3],0,0,0,0,0),
\]
with open-stratum term \(\Q_U[3]\). After choosing splittings, the point terms are one-dimensional and the extension class is encoded in the corresponding gluing parameter.
\end{proposition}

\begin{proof}
Applying \(\mu\) to the short exact sequence gives an extension zig-zag of the indicated endpoint objects. In the ordinary double point case, the singular contribution is rank one, so the point terms are one-dimensional. The distinction between split and non-split extensions is therefore carried by the gluing class rather than by the ambient dimensions.
\end{proof}

\subsection{The corrected perverse object}

\begin{proposition}
\label{prop:appendix-corrected}
The corrected perverse object
\[
\mathcal P:=\Cone(\var_F)[-1]
\]
has zig-zag
\[
\mu(\mathcal P)\cong (\Q_U[3],\Q,\Q,0,\id,0).
\]
\end{proposition}

\begin{proof}
This is Proposition~\ref{prop:full-zigzag-P}. The corrected object is the distinguished non-split extension of \(i_*\Q_{\{p\}}\) by \(IC_{X_0}\), so its zig-zag is the corresponding nontrivial extension zig-zag of the endpoint objects.
\end{proof}

\subsection{Duality}

\begin{proposition}
\label{prop:appendix-duality}
The zig-zags of \(IC_{X_0}\), \(i_*\Q_{\{p\}}\), and \(\mathcal P\) are self-dual under the zig-zag duality functor of \cite{RahmanATMP}.
\end{proposition}

\begin{proof}
For \(IC_{X_0}\) and \(i_*\Q_{\{p\}}\), this follows from the explicit zig-zag calculations above. For \(\mathcal P\), it follows from Theorem~\ref{thm:verdier-self-duality-paper2}.
\end{proof}

\subsection{Multi-node direct sums}

\begin{proposition}
\label{prop:appendix-multinode}
For a finite set of isolated nodes \(\Sigma=\{p_1,\dots,p_r\}\), the direct-sum point-supported object
\[
\bigoplus_{k=1}^r i_{k*}\Q_{\{p_k\}}
\]
has zig-zag shadow given componentwise by the direct sum of the corresponding rank-one point zig-zags.
\end{proposition}

\begin{proof}
Each \(i_{k*}\Q_{\{p_k\}}\) contributes the rank-one point zig-zag
\[
(0,\Q,\Q,0,\id,0),
\]
and the direct sum is formed componentwise.
\end{proof}

\subsection{Table of standard zig-zags}

\begin{center}
\footnotesize
\renewcommand{\arraystretch}{1.2}
\setlength{\tabcolsep}{4pt}
\captionof{table}{Standard zig-zags in the ordinary double point case.}
\begin{tabular}{|C{0.23\textwidth}|C{0.39\textwidth}|C{0.28\textwidth}|}
\hline
\textbf{Object} & \textbf{Zig-zag} & \textbf{Comments} \\
\hline
\(IC_{X_0}\) &
\((\Q_U[3],0,0,0,0,0)\) &
minimal extension \\
\hline
\(i_*\Q_{\{p\}}\) &
\((0,\Q,\Q,0,\id,0)\) &
point-supported rank-one object \\
\hline
\(\mathcal P:=\Cone(\var_F)[-1]\) &
\((\Q_U[3],\Q,\Q,0,\id,0)\) &
distinguished corrected non-split class \\
\hline
\(\bigoplus_{k=1}^r i_{k*}\Q_{\{p_k\}}\) &
\(\bigoplus_{k=1}^r (0,\Q,\Q,0,\id,0)\) &
multi-node local shadow \\
\hline
\end{tabular}
\end{center}

\medskip

In this compressed rank-one notation, split and non-split extensions may have the same ambient zig-zag shape. They are distinguished not by the raw ambient dimensions, but by the underlying extension class, equivalently by the associated gluing parameter in a matrix-form presentation.

\medskip

\begin{proposition}
\label{prop:compressed-zigzag-not-complete}
Let
\[
0\to K'\to K\to K''\to 0
\]
be a short exact sequence in \(\Perv(X_0;\Q)\), and let
\[
\mu(K')=(L',A',B',\alpha',\beta',\gamma'),
\qquad
\mu(K'')=(L'',A'',B'',\alpha'',\beta'',\gamma'')
\]
be the corresponding MacPherson--Vilonen zig-zags. Then the zig-zag of \(K\) is determined by an extension of the endpoint zig-zags together with the induced gluing data in the corresponding maps. In particular, the compressed ambient zig-zag shape of \(K\), that is, the tuple obtained by recording only the open-stratum term and the ambient point terms, does not in general determine the isomorphism class of \(K\). Equivalently, two non-isomorphic extension objects may have the same compressed ambient zig-zag shape while differing in their extension class.
\end{proposition}

\begin{proof}
Applying the MacPherson--Vilonen zig-zag functor \(\mu\) to a short exact sequence
\[
0\to K'\to K\to K''\to 0
\]
produces an extension of the corresponding zig-zag data. In particular, the open-stratum term and the point terms fit into exact sequences
\[
0\to L'\to L\to L''\to 0,
\qquad
0\to A'\to A\to A''\to 0,
\qquad
0\to B'\to B\to B''\to 0,
\]
where
\[
\mu(K)=(L,A,B,\alpha,\beta,\gamma).
\]

After choosing splittings of the underlying vector-space extensions, one may identify
\[
A\cong A'\oplus A'',
\qquad
B\cong B'\oplus B'',
\]
so that the induced map \(\beta\) is represented by a block matrix whose off-diagonal term records the extension class. Thus the ambient dimensions of \(L\), \(A\), and \(B\), and even the resulting compressed zig-zag shape, do not by themselves determine the isomorphism class of \(K\). The missing information is precisely the gluing data carried by the induced maps, equivalently the extension class. Therefore two extension objects may have the same compressed ambient zig-zag shape while still be non-isomorphic.
\end{proof}

Table~2 illustrates Proposition~\ref{prop:compressed-zigzag-not-complete} in the ordinary double point case, where the split extension and the corrected non-split extension have the same compressed ambient zig-zag shape but are distinguished by their extension class.

\begin{center}
\scriptsize
\renewcommand{\arraystretch}{1.2}
\setlength{\tabcolsep}{3pt}
\captionof{table}{Extension templates in zig-zag form.}
\begin{tabular}{|C{0.33\textwidth}|C{0.33\textwidth}|C{0.24\textwidth}|}
\hline
\textbf{Object} & \textbf{Zig-zag} & \textbf{Comments} \\
\hline
\parbox[c]{\linewidth}{\centering split extension\\[-2pt]
{\tiny \(0\to IC_{X_0}\to IC_{X_0}\oplus i_*\Q_{\{p\}}\to i_*\Q_{\{p\}}\to 0\)}} &
\((\Q_U[3],\Q,\Q,0,\id,0)\) &
trivial extension class \\
\hline
\parbox[c]{\linewidth}{\centering general extension \(E\)\\[-2pt]
{\tiny \(0\to IC_{X_0}\to E\to i_*\Q_{\{p\}}\to 0\)}} &
\(\left(\Q_U[3],\,A\oplus\Q,\,B\oplus\Q,\,\alpha_E,\,
\begin{psmallmatrix}\beta & u\\[2pt] 0 & 1\end{psmallmatrix},\,\gamma_E\right)\) &
Here \(A\) and \(B\) denote the point terms in a general endpoint zig-zag
\((\Q_U[3],A,B,\alpha,\beta,\gamma)\). For the specific object \(IC_{X_0}\) considered in this appendix,
one has \(A=B=0\), so the ordinary double point specialization collapses to the rank-one form listed
for \(\mathcal P\). The parameter \(u\in B\) records the extension class modulo \(\operatorname{Im}\beta\). \\
\hline
\parbox[c]{\linewidth}{\centering corrected non-split extension \(\mathcal P\)\\[-2pt]
{\tiny \(0\to IC_{X_0}\to \mathcal P\to i_*\Q_{\{p\}}\to 0\)}} &
\((\Q_U[3],\Q,\Q,0,\id,0)\) &
distinguished nontrivial self-dual class \\
\hline
\end{tabular}
\end{center}
%-----------------------------
%
%           Bibliography
%
\printbibliography
\end{document}